\documentclass[a4paper,11pt]{article}

\usepackage[T1]{fontenc}
\usepackage{lmodern}

\usepackage{dsfont}
\usepackage{float}
\usepackage[latin1]{inputenc}
\usepackage{amsmath}
\usepackage{mathtools}
\usepackage{amsthm}
\usepackage{amssymb}
\usepackage{mathrsfs}
\usepackage{graphicx}
\usepackage[all]{xy}
\usepackage{hyperref}

\usepackage{stmaryrd}
\usepackage{caption}

\usepackage{abstract} 

\newtheorem{thm}{Theorem}[section]
\newtheorem{cor}[thm]{Corollary}
\newtheorem{claim}[thm]{Claim}
\newtheorem{fact}[thm]{Fact}

\newtheorem{lemma}[thm]{Lemma}
\newtheorem{prop}[thm]{Proposition}
\newtheorem{lem}[thm]{Lemma}

\theoremstyle{definition}
\newtheorem{definition}[thm]{Definition}
\newtheorem{obs}[thm]{Observation}

\newtheorem{remark}[thm]{Remark}

\newtheorem*{thmA}{Theorem A}
\newtheorem*{thmB}{Theorem B}

\newtheorem*{corC}{Corollary C}
\newtheorem*{thmD}{Theorem D}
\newtheorem*{thmE}{Theorem E}
\newtheorem*{thmF}{Theorem F}

\newcommand{\bbZ}{{\mathbb Z}}

\newcommand{\cG}{{\mathcal G}}

\def\rquotient#1#2{%
	\makeatletter
	\raise.3ex\hbox{$#1$}/\lower.3ex\hbox{$#2$}%
	\makeatother
}	

\makeatletter
\newcommand{\subjclass}[2][2010]{%
	\let\@oldtitle\@title%
	\gdef\@title{\@oldtitle\footnotetext{#1 \emph{Mathematics subject classification.} #2}}%
}
\newcommand{\keywords}[1]{%
	\let\@@oldtitle\@title%
	\gdef\@title{\@@oldtitle\footnotetext{\emph{Key words and phrases.} #1.}}%
}
\makeatother

\newcommand{\Address}{{
		\bigskip
		\small
		
		\textsc{D\'epartement de Math\'ematiques B\^atiment 307, Facult\'e des Sciences d'Orsay, Universit\'e Paris-Sud, F-91405 Orsay Cedex, France.}\par\nopagebreak
		\textit{E-mail address}: \texttt{anthony.genevois@math.u-psud.fr}
		
		\medskip
		
		\textsc{Department of Mathematics and the Maxwell Institute for Mathematical Sciences, Heriot-Watt University, Riccarton, EH14 4AS Edinburgh, United Kingdom.}\par\nopagebreak
		\textit{E-mail address}: \texttt{alexandre.martin@hw.ac.uk}
		
}}

\title{Automorphisms of graph products of groups from a geometric perspective}
\date{\today}
\author{Anthony Genevois and Alexandre Martin}


\begin{document}

\maketitle

\begin{abstract}
This article studies the structure of the automorphism groups of general graph products of groups. We give a complete characterisation of the automorphisms that preserve the set of conjugacy classes of vertex groups for arbitrary graph products. Under mild conditions on the underlying graph, this allows us to provide a simple set of generators for the automorphism groups of graph products of \textit{arbitrary} groups. We also obtain information  about the geometry of the automorphism groups of such graph products: lack of property (T), acylindrical hyperbolicity.

The approach in this article is geometric and relies on the action of graph products of groups on certain complexes with a particularly rich combinatorial geometry. The first such complex is a particular Cayley graph of the graph product that  has a \textit{quasi-median} geometry, a combinatorial geometry reminiscent of (but more general than) CAT(0) cube complexes. The second (strongly related)  complex used is the Davis complex of the graph product, a CAT(0) cube complex that also has a structure of right-angled building. 
\end{abstract}

%

\tableofcontents

\newpage

\section{Introduction and main results}

Graph products of groups, which have been introduced by  Green in \cite{GreenGP}, define a class of group products that, loosely speaking, interpolates between free and direct products. For a simplicial graph $\Gamma$ and a collection of groups $\mathcal{G}=\{ G_v \mid v \in V(\Gamma) \}$ indexed by the vertex set $V(\Gamma)$ of $\Gamma$, the \emph{graph product} $\Gamma \mathcal{G}$ is defined as the quotient
$$\left( \underset{v \in V(\Gamma)}{\ast} G_v \right) / \langle \langle gh=hg, \ h \in G_u, g \in G_v, \{ u,v \} \in E(\Gamma) \rangle \rangle,$$
where $E(\Gamma)$ denotes the edge set of $\Gamma$. The two extreme situations where $\Gamma$ has no edge and where $\Gamma$ is a complete graph respectively correspond to the free product and the direct sum of the groups belonging to the collection $\mathcal{G}$. Graph products include two intensively studied families of groups: right-angled Artin groups  and right-angled Coxeter groups. Many articles have been dedicated to the study of the automorphism groups of these particular examples of graph products. In particular, the automorphism groups of right-angled Coxeter groups have been intensively studied in relation with the famous rigidity problem  for Coxeter groups, see for instance \cite{RACGrigidity}. 

Beyond these two cases, the automorphism groups of general graph products of groups are poorly understood. Most of the literature on this topic imposes very strong conditions on the graph products involved, either on the underlying graph (as in the case of the automorphisms groups of free products \cite{OutSpaceFreeProduct, HorbezHypGraphsForFreeProducts, HorbezTitsAlt}) or on the vertex groups (most of the case, they are required to be abelian or even cyclic  \cite{AutGPabelianSet, AutGPabelian, AutGPSIL, RuaneWitzel}). Automorphism groups of graph products of more general groups (and their subgroups) are essentially uncharted territory. For instance, the following general problem is still unsolved:\\

\textbf{General Problem.} Find a natural / simple generating set for the automorphism group of a general graph product of groups.\\


The first result in that direction is the case of right-angled Artin groups or right-angled Coxeter groups, solved by Servatius \cite{RAAGServatius} and Laurence \cite{RAAGgenerators}. More recently, Corredor--Guttierez described a generating set for automorphism groups of graph products of cyclic groups \cite{AutGPabelianSet}, using previous work of Guttierez--Piggott--Ruane \cite{AutGPabelian}.  Beyond these cases however, virtually nothing is known about the automorphism group of a graph product.

Certain elements in the generating sets of a right-angled Artin groups naturally generalise to more general graph products, and we take a moment to mention them as they play an important role in the present work:
\begin{itemize}
	\item For an element $g\in \Gamma\cG$,  the \textit{inner automorphism} $\iota(g)$ is defined by $$\iota(g): \Gamma\cG \rightarrow \Gamma\cG, ~~x \mapsto gxg^{-1}.$$
	\item Given an isometry $\sigma : \Gamma \to \Gamma$ and a collection of isomorphisms $\Phi = \{ \varphi_u : G_u \to G_{\sigma(u)} \mid u \in V(\Gamma) \}$, the \emph{local automorphism} $(\sigma, \Phi)$ is the automorphism of $\Gamma \mathcal{G}$ induced by $$\left\{ \begin{array}{ccc} \bigcup\limits_{u \in V(\Gamma)} G_u & \to & \Gamma \mathcal{G} \\ g & \mapsto & \text{$\varphi_u(g)$ if $g \in G_u$} \end{array} \right. .$$ For instance, in the specific case of right-angled Artin groups, graphic automorphisms (i.e. automorphisms of $\Gamma\cG$ induced by a graph automorphism of $\Gamma$) and inversions \cite{ServatiusCent} are local automorphisms.
	\item Given a vertex $u \in V(\Gamma)$, a connected component $\Lambda$ of $\Gamma \backslash \mathrm{star}(u)$ and an element $h \in G_u$, the \emph{partial conjugation} $(u, \Lambda,h)$ is the automorphism of $\Gamma \mathcal{G}$ induced by $$\left\{ \begin{array}{ccc} \bigcup\limits_{u \in V(\Gamma)} G_u & \to & \Gamma \mathcal{G} \\ g & \mapsto & \left\{ \begin{array}{cl} g & \text{if $g \notin \langle \Lambda \rangle$} \\ hgh^{-1} & \text{if $g \in \langle \Lambda \rangle$} \end{array} \right. \end{array} \right. .$$ 
	Notice that an inner automorphism of $\Gamma \mathcal{G}$ is always a product of partial conjugations. 
\end{itemize}


The goal of this article is to describe the structure (and provide a generating set) for much larger classes of graphs products of groups by adopting a new geometric perspective. In a nutshell, the strategy is to consider the action of graph products $\Gamma\cG$ on an appropriate space and to show that  this action can be extended  to an action of $\mathrm{Aut}(\Gamma\cG)$ on $X$, in order to exploit the geometry of this action.
Such a `rigidity' phenomenon appeared for instance in the work of Ivanov on the action of mapping class groups of hyperbolic surfaces on their curve complexes: Ivanov showed that an automorphism of the mapping class group induces an automorphism of the underlying curve complex \cite{IvanovAut}. Another example is given by the Higman group: in \cite{MartinHigman}, the second author computed the automorphism group of the Higman group $H_4$  by first extending the action of $H_4$ on a CAT(0) square complex naturally associated to its standard presentation to an action of $\mathrm{Aut}(H_4)$. In this article, we construct such rigid actions for large classes of graph products of groups. The results from this article vastly generalise earlier results obtained by the authors in a previous  version (not intended for publication, but still available on the arXiv as \cite{GPcycle}).\\

We now present the main results of this article. \textbf{For the rest of this introduction, we fix a  finite simplicial graph $\Gamma$ and a collection $\cG$ of groups  indexed by  $V(\Gamma)$. }

\paragraph{The subgroup of conjugating automorphisms.} When studying automorphisms of graph products of groups, an important subgroup consists of those automorphisms that send vertex groups to conjugates of vertex groups. This subgroup already appears for instance in the work of Tits \cite{TitsAutCoxeter}, Corredor--Gutierrez \cite{AutGPabelianSet}, and Charney--Gutierez--Ruane \cite{AutGPabelian}. A description of this subgroup was only available for right-angled Artin groups and other graph products of \textit{cyclic} groups by work of Laurence \cite{RAAGgenerators}. 

A central result of this paper is a complete characterisation of this subgroup under \textit{no restriction on the vertex groups or the underlying graph}.  More precisely, let us call a \emph{conjugating automorphism} of $\Gamma \mathcal{G}$  an automorphism $\varphi$ of $\Gamma \mathcal{G}$ such that, for every vertex group $G_v \in \mathcal{G}$, there exists a vertex group  $G_w \in \mathcal{G}$ and an element $g \in \Gamma \mathcal{G}$ such that $\varphi(G_v)=gG_wg^{-1}$. We prove the following:

\begin{thmA}
	\emph{	The subgroup of conjugating automorphisms of $\Gamma \mathcal{G}$ is exactly the subgroup of $ \mathrm{Aut}(\Gamma\cG)$ generated by the local automorphisms and the partial conjugations.  }
\end{thmA}

\paragraph{Generating set and algebraic structure.} With this characterisation of conjugating automorphisms at our disposal, we are able to completely describe the automorphism group of large classes of graph products, and in particular to give  a generating set for such automorphism groups. To the authors' knowledge, this result represents the first results on the algebraic structure of automorphism groups of graph products of general (and in particular non-abelian) groups. 

\begin{thmB}
 \emph{If $\Gamma$ is a finite connected simplicial graph  of girth at least $5$ and without vertices of valence $<2$, then $\mathrm{Aut}(\Gamma\cG)$ is generated by the partial conjugations and the local automorphisms.}
\end{thmB}

This description of the automorphism group in Theorem B simplifies further in the case where  $\Gamma$ is in addition assumed not to contain any separating star (following \cite{RAAGatomic}, a finite connected graph without vertex of valence $<2$, whose girth is at least $5$,  and that does not contain separating stars is called \textit{atomic}), then we get the following decomposition: 

\begin{corC}
	 \emph{If $\Gamma$ is an atomic graph, then}

	$$\mathrm{Aut}(\Gamma\mathcal{G})
	\simeq \Gamma\mathcal{G} \rtimes \left( \left( \prod\limits_{v \in \Gamma} \mathrm{Aut}(G_v) \right) \rtimes \mathrm{Sym}(\Gamma \mathcal{G}) \right), $$
	\emph{where $\mathrm{Sym}(\Gamma \mathcal{G})$ is an explicit subgroup of the automorphism group of  $\Gamma$.}
\end{corC}

Actually, we obtain a stronger statement characterising isomorphisms between graph products of groups (see Theorem \ref{thm:ConjIsom}), which in the case of right-angled Coxeter groups is strongly related to the so-called \textit{strong rigidity} of these groups, and to the famous isomorphism problem for general Coxeter groups, see \cite{RACGrigidity}. 

It should be noted that while the previous theorems impose conditions on the underlying graph, it can be used to obtain information about more general Coxeter or Artin groups. Indeed, if the vertex groups in our graph product are (arbitrary) Coxeter groups, then the resulting graph product is again a Coxeter group (with a possibly much wilder underlying graph), and Corollary C can thus be interpreted as a form of strong rigidity of these Coxeter groups \textit{relative to their vertex groups}: up to conjugation \textit{and automorphisms of the vertex groups}, an automorphism of the graph product comes from a (suitable) isometry of the underlying graph.\\

The  explicit computation in Corollary C can be used to study the subgroups of such automorphism groups. In particular, since satisfying the Tits alternative is a property stable under graph products \cite{AM} and under extensions, one can deduce from Corollary C a combination theorem for the Tits Alternative for such automorphism groups. \\

We mention also an application of this circle of ideas to the study of automorphism groups of  graph products of finite groups, with no requirement on the underlying graph. The following result was only known for graph products of \textit{cyclic} groups by work of Corredor--Gutierrez \cite{AutGPabelianSet}:

\begin{thmD}
\emph{If all the groups of $\cG$ are finite, then the subgroup of conjugating automorphisms of $\Gamma\cG$  has finite index in $\mathrm{Aut}(\Gamma \mathcal{G})$. }
\end{thmD}

As an interesting application, we are able to determine precisely when a graph product of finite groups has a finite outer automorphism group. See Corollary \ref{cor:OutFinite} for a precise statement.


\paragraph{Geometry of the automorphism group.} While the previous results give us information about the algebraic structure of the automorphism groups of graph products, the geometric point of view used in this article also allows us to obtain some information about their geometry. \\

The first property we investigate is the  notion of \textit{acylindrical hyperbolicity}
introduced by Osin in \cite{OsinAcyl}, which unifies several known classes of groups with `negatively curved' features such as relatively hyperbolic groups and mapping class groups (we refer to \cite{OsinSurvey} for more information). One of the most striking consequences of the acylindrical hyperbolicity of a group is its \emph{SQ-universality} \cite{DGO}, that is, every countable group embeds into a quotient of the group we are looking at. 
Loosely speaking, such groups are thus very far from being simple. 

For general graph products, we obtain the following: 

\begin{thmE}
		\emph{If $\Gamma$ is an atomic graph and if $\mathcal{G}$ is collection of \textit{finitely generated} groups, then $\mathrm{Aut}(\Gamma\cG)$ is acylindrically hyperbolic.}
\end{thmE}

Let us mention that prior to this result, very little was known about the acylindrical hyperbolicity of the automorphism group of a graph product, even in the case of right-angled Artin groups. At one end of the RAAG spectrum, $\mathrm{Aut}(\bbZ^n)= \mathrm{GL}_n(\bbZ)$ is a higher rank lattice for $n \geq 3$, and thus does not have non-elementary actions on hyperbolic spaces by a recent result of Haettel \cite{HaettelRigidity}. The situation is less clear for $\mathrm{Aut}(F_n)$: while it is known that $\mathrm{Out}(F_n)$ is acylindrically hyperbolic for $n \geq 2$ \cite{BestvinaFeighnOutFnHyp}, the case of  $\mathrm{Aut}(F_n)$ seems to be open. For right-angled Artin groups whose outer automorphism group is finite, such as right-angled Artin group over atomic graphs, the problem boils down to the question of the acylindrical hyperbolicity of the underlying group, for which a complete answer is known \cite{MinasyanOsinTrees}. \\

Another property of a more geometric nature group that can be investigated from our perspective is \textit{Kazhdan's property (T)}. 
Property (T) for a group imposes for instance strong restrictions on the possible homomorphisms starting from that group and plays a fundamental role in several rigidity statements, including  Margulis' superrigidity. 
We only possess a fragmented picture of the status of property (T) for automorphism groups of right-angled Artin groups. At one end of the RAAG spectrum, $\mathrm{Aut}(\bbZ^n)= \mathrm{GL}_n(\bbZ)$ is known to have property (T) for $n \geq 3$. In the opposite direction, it is known that $\mathrm{Aut}(F_n)$ does not have property (T) for $n =2$ and $3$ \cite{McCoolPropTaut, GrunewaldLubotzkyAutF, BogopolskiVikentievAutF}; that it has property (T) for $n \geq 5$ by very recent results of \cite{Aut5PropT, AutFnPropT}; and the case $n=4$ is still open. About more general right-angled Artin groups, a few general criteria can be found in \cite{RAAGandT, VastAutRAAG}. (See also \cite{VastAutRACG} for right-angled Coxeter groups.)
We obtain the following result:
\begin{thmF}
\emph{If $\Gamma$ is an atomic graph, then $\mathrm{Aut}(\Gamma\cG)$ does not have Property (T).}
\end{thmF}


We emphasize that this result does not assume any knowledge of the vertex groups of  the graph product, or the size of its outer automorphism group. In particular, by allowing vertex groups to be arbitrary right-angled Artin groups, this result provides a very large class of right-angled Artin groups whose automorphism groups do not have property (T).\\

\paragraph{Structure of the paper.} Let us now detail the strategy and structure of this article. In Section 2, we recall a few general definitions and statements about graph products, before introducing the two main complexes studied in this paper:  the Davis complex of a graph product of groups $\Gamma\cG$, and a certain Cayley graph $X(\Gamma, \cG)$ of the graph product that has a particularly rich combinatorial geometry (namely, a \textit{quasi-median} geometry). Ideally, one would like to show that the action of the graph product on one of these complexes extends to an action of the automorphism group. However, this does not hold in general: in the case of right-angled Artin groups for instance, the presence of transvections shows that conjugacy classes of vertex groups are not preserved by automorphisms in general.  In Section 3, a different action is considered, namely the action of $\Gamma\cG$ on the \textit{transversality graph} associated to the graph  $X(\Gamma, \cG)$, and we show that this action extends to an action of Aut$(\Gamma\cG)$. (This transversality graph turns out to be naturally isomorphic to the intersection graph of parallelism classes of hyperplanes in the Davis complex.) This action allows us to prove  the central  result of our article: the characterisation of conjugating automorphisms stated in Theorem A.   The algebraic structure of certain automorphism groups is also proved in this section (Theorems B and D). Finally, Section 4 focuses on the case of graph products of groups over atomic graphs. We first prove that the action of the graph product on its Davis complex extends to an action of its automorphism group. Such a rich action on a CAT(0) cube complex is then used to prove Theorems E and F.\\

\paragraph{The point of view of quasi-median graphs.} We take a moment to justify the point of view adopted in this article, and in particular the central role played by the quasi-median geometry of some Cayley graph of  a graph product. This is a very natural object associated to the group, and its geometry turns out to be both similar and simpler than that of the (perhaps more familiar) Davis complex. Quasi-median graphs have been studied in great detail by the first author \cite{Qm}. However, we wish to emphasize that \textbf{we provide in this article self-contained proofs of all the combinatorial/geometric results about this graph, in order to avoid relying on the (yet unpublished, at the time of writing) manuscript} \cite{Qm}. In particular, no prerequisite on quasi-median geometry is needed to read this article.

Let us explain further the advantages of this (quasi-median) graph over the Davis complex.  First, the geodesics of this graph encode the normal forms of group elements, which makes its geometry more natural and easier to work with; see Section \ref{section:CubicalLikeGeom} for more details. Moreover, although a quasi-median graph is not the $1$-skeleton of a CAT(0) cube complex, it turns out to have essentially the same type of geometry.  More precisely, \emph{hyperplanes} may be defined in quasi-median graphs in a similar fashion, and so that the geometry reduces to the combinatorics of hyperplanes, as for CAT(0) cube complexes; see for instance Theorem \ref{thm:QmHyperplanes} below. Roughly speaking, quasi-median graphs may  be thought of as `almost' CAT(0) cube complexes in which hyperplanes cut the space into at least two pieces but possibly more.  The analogies between these classes of spaces go much further, and we refer to \cite[Section 2]{Qm} for a dictionary between concepts/results in CAT(0) cube complexes  and their quasi-median counterparts. Hyperplanes in the quasi-median graph turn out to be easier to work with than hyperplanes in the Davis complex, due to the absence of parallel hyperplanes, which makes some of the arguments simpler and cleaner. Hyperplanes in this quasi-median graph are closely related to the \emph{tree-walls} of the Davis complex introduced by M. Bourdon in \cite{Bourdon} for certain graph products of groups, and used in a previous version of this article (not intended for publication) \cite{GPcycle}.

  Quasi-median graphs provide a convenient combinatorial framework that encompasses and unifies many of the tools used to study graph products until now:  the normal forms proved in E. Green's thesis \cite{GreenGP} (see also \cite{GPvanKampen} for a more geometric approach), 
 the action on a right-angled building with an emphasis on the combinatorics of certain subspaces (tree-walls) \cite{Bourdon, TW, CapraceTW, UnivRAbuildings}, the  action on a CAT(0) cube complex (Davis complex), etc.
An indirect goal of this article is thus to convince the reader that the quasi-median graph associated to a graph product of groups provides a rich and natural combinatorial setting to study this group, and ought to be investigated further. Finally, let us mention that decomposing non-trivially as a graph product can be characterised by the existence of an appropriate action on a quasi-median graph, see \cite[Corollary 10.57]{Qm}. Thus, quasi-median geometry is in a sense `the' natural geometry associated with graph products.


\medskip \noindent


\paragraph{Acknowledgements.} We are grateful to Olga Varghense for having communicated to us a proof of Corollary \ref{cor:OutFinite}, replacing a wrong argument in a preliminary version. The first author thanks the university of Vienna for its hospitality in December 2015, where this project originated; and from March to May 2018 where he was a visitor, supported by the Ernst Mach Grant ICM-2017-06478 under the supervision of Goulnara Arzhantseva. The second author thanks the MSRI for its support during the program `Geometric Group Theory' in 2016, Grant  DMS-1440140, and the Isaac Newton Institute for its support during the program `Non-positive curvature group actions and cohomology' in 2017, EPSRC Grant  EP/K032208/1, during which part of this research was conducted. The second author was partially supported by the ERC grant 259527,   the Lise Meitner FWF project M 1810-N25, and the EPSRC New Investigator Award EP/S010963/1.


\section{Geometries associated to graph products of groups}\label{sec:geom}

\subsection{Generalities about graph products}\label{section:GP}\label{sec:generalities}

Given a simplicial graph $\Gamma$, whose set of vertices is denoted by $V(\Gamma)$, and a collection of groups $\mathcal{G}=\{ G_v \mid v \in V(\Gamma) \}$ (the \textit{vertex groups}),  the \emph{graph product} $\Gamma \mathcal{G}$ is defined as the quotient
\begin{center}
$\left( \underset{v \in V(\Gamma)}{\ast} G_v \right) / \langle \langle gh=hg, \ h \in G_u, g \in G_v, \{ u,v \} \in E(\Gamma) \rangle \rangle$,
\end{center}
where $E(\Gamma)$ denotes the set of edges of $\Gamma$. Loosely speaking, it is obtained from the disjoint union of all the $G_v$'s, called the \emph{vertex groups}, by requiring that two adjacent vertex groups  commute. Notice that, if $\Gamma$ has no edges, $\Gamma \mathcal{G}$ is the free product of the groups of $\mathcal{G}$; on the other hand, if $\Gamma$ is a complete graph, then $\Gamma \mathcal{G}$ is the direct sum of the groups of $\mathcal{G}$. Therefore, a graph product may be thought of as an interpolation between free and direct products. Graph products also include two classical families of groups: If all the vertex groups are infinite cyclic, $\Gamma \mathcal{G}$ is known as a \emph{right-angled Artin group}; and if all the vertex groups are cyclic of order two, then $\Gamma \mathcal{G}$ is known as a \emph{right-angled Coxeter group}.

\medskip \noindent
\textbf{Convention.} In all the article, we will assume for convenience that the groups of $\mathcal{G}$ are non-trivial. Notice that it is not a restrictive assumption, since a graph product with some trivial factors can be described as a graph product over a smaller graph all of whose factors are non-trivial. 

\medskip \noindent
If $\Lambda$ is an \emph{induced} subgraph of $\Gamma$ (ie., two vertices of $\Lambda$ are adjacent in $\Lambda$ if and only if they are adjacent in $\Gamma$), then the subgroup, which we denote by $\langle \Lambda \rangle$, generated by the vertex groups corresponding to the vertices of $\Lambda$ is naturally isomorphic to the graph product $\Lambda \mathcal{G}_{|\Lambda}$, where $\mathcal{G}_{|\Lambda}$ denotes the subcollection of $\mathcal{G}$ associated to the set of vertices of $\Lambda$. This observation is for instance a consequence of the normal form described below.\\

\textbf{For the rest of Section \ref{sec:geom}, we fix a  finite simplicial graph $\Gamma$ and a collection $\cG$ of groups  indexed by  $V(\Gamma)$. }

\paragraph{Normal form.} A \emph{word} in $\Gamma \mathcal{G}$ is a product $g_1 \cdots g_n$ where $n \geq 0$ and where, for every $1 \leq i \leq n$, $g_i$ belongs to $G_i$ for some $G_i \in \mathcal{G}$; the $g_i$'s are the \emph{syllables} of the word, and $n$ is the \emph{length} of the word. Clearly, the following operations on a word does not modify the element of $\Gamma \mathcal{G}$ it represents:
\begin{itemize}
	\item[(O1)] delete the syllable $g_i=1$;
	\item[(O2)] if $g_i,g_{i+1} \in G$ for some $G \in \mathcal{G}$, replace the two syllables $g_i$ and $g_{i+1}$ by the single syllable $g_ig_{i+1} \in G$;
	\item[(O3)] if $g_i$ and $g_{i+1}$ belong to two adjacent vertex groups, switch them.
\end{itemize}
A word is \emph{reduced} if its length cannot be shortened by applying these elementary moves. Given a word $g_1 \cdots g_n$ and some $1\leq i<n$, if  the vertex group associated to $g_i$ is adjacent to each of the vertex groups of $g_{i+1}, ..., g_n$, then the words $g_1 \cdots g_n$ and $g_1\cdots g_{i-1} \cdot g_{i+1} \cdots g_n \cdot g_i$ represent the same element of $\Gamma\cG$; We say that $g_i$ \textit{shuffles to the right}. Analogously, one can define the notion of a syllable shuffling to the left. If $g=g_1 \cdots g_n$ is a reduced word and $h$ is a syllable, then a reduction of the product $gh$ is given by
\begin{itemize}
	\item $g_1 \cdots g_n$ if $h=1$;
	\item $g_1 \cdots g_{i-1} \cdot g_{i+1} \cdots g_n$ if $g_i$ and $h$ belong to the same vertex group,  $g_i$ shuffles to the right and $g_i=h^{-1}$;
	\item $g_1 \cdots g_{i-1} \cdot g_{i+1} \cdots g_n \cdot (g_ih)$ if $g_i$ and $h$ belong to the same vertex group, $g_i$ shuffles to the right, $g_i \neq h^{-1}$ and $(g_ih)$ is thought of as a single syllable.
\end{itemize}
In particular, every element of $\Gamma \mathcal{G}$ can be represented by a reduced word, and this word is unique up to applying the operation (O3). It is worth noticing that a reduced word has also minimal length with respect to the generating set $\bigcup\limits_{u\in V(\Gamma)} G_u$. We refer to \cite{GreenGP} for more details (see also \cite{GPvanKampen} for a more geometric approach). We will also need the following definition:

\begin{definition}
Let $g \in \Gamma\cG$. The \emph{head} of $g$, denoted by $\mathrm{head}(g)$, is the collection of the first syllables appearing in the reduced words representing $g$. Similarly, the \emph{tail} of $g$, denoted by $\mathrm{tail}(g)$, is the collection of the last syllables appearing in the reduced words representing $g$. 
\end{definition}

\paragraph{Some vocabulary.} We conclude this paragraph with a few definitions about graphs used in the article. 
\begin{itemize}
	\item A subgraph $\Lambda \subset \Gamma$ is a \emph{join} if there exists a partition $V(\Lambda)= A \sqcup B$, where $A$ and $B$ are both non-empty, such that any vertex of $A$ is adjacent to any vertex of $B$.
	\item Given a vertex $u \in V(\Gamma)$, its \emph{link}, denoted by $\mathrm{link}(u)$, is the subgraph generated by the neighbors of $u$.
	\item More generally, given a subgroup $\Lambda \subset \Gamma$, its \emph{link}, denoted by $\mathrm{link}(\Lambda)$, is the subgraph generated by the vertices of $\Gamma$ which are adjacent to all the vertices of~$\Lambda$.
	\item Given a vertex $u \in V(\Gamma)$, its \emph{star}, denoted by $\mathrm{star}(u)$, is the subgraph generated by $\mathrm{link}(u) \cup \{u\}$.
	\item More generally, given a subgraph $\Lambda \subset \Gamma$, its \emph{star}, denoted by $\mathrm{star}(\Lambda)$, is the subgraph generated by $\mathrm{link}(\Lambda) \cup \Lambda$. 
\end{itemize}

\subsection{The quasi-median graph associated to a graph product of groups}\label{section:CubicalLikeGeom}

\noindent
This section is dedicated to the geometry of the following Cayley graph  of $\Gamma\cG$:
$$X(\Gamma, \mathcal{G}) : = \mathrm{Cayl} \left( \Gamma \mathcal{G}, \bigcup\limits_{u \in V(\Gamma)} G_u \backslash \{1 \} \right),$$
ie., the graph whose vertices are the elements of the groups $\Gamma \mathcal{G}$ and whose edges link two distinct vertices $x,y \in \Gamma \mathcal{G}$ if $y^{-1}x$ is a non-trivial element of some vertex group. Like in any Cayley graph, edges of $X(\Gamma, \mathcal{G})$ are labelled by generators, namely by elements of vertex groups. By extension, paths in $X(\Gamma, \mathcal{G})$ are naturally labelled by words of generators. In particular, geodesics in $X(\Gamma, \mathcal{G})$ correspond to words of minimal length. More precisely:

\begin{prop}\label{prop:distinGP}
Fix two vertices $g,h \in X(\Gamma, \mathcal{G})$. If $s_1 \cdots s_n$ is a reduced word representing $g^{-1}h$, then 
$$g, \ gs_1, \ gs_1s_2, \ldots, gs_1 \cdots s_{n-1}, \ gs_1 \cdots s_{n-1}s_n = h$$
define a geodesic in $X(\Gamma, \mathcal{G})$ from $g$ to $h$. Conversely, if $s_1, \ldots, s_n$ is the sequence of elements of $\Gamma \mathcal{G}$ labelling the edges of a geodesic in $X(\Gamma, \mathcal{G})$ from $g$ to $h$, then $s_1 \cdots s_n$ is a reduced word representing $g^{-1}h$. As a consequence, the distance in $X(\Gamma, \mathcal{G})$ between $g$ and $h$ coincides with the length $|g^{-1}h|$ of any reduced word representing $g^{-1}h$.
\end{prop}

\paragraph{The Cayley graph as a complex of prisms.} The first thing we want to highlight is that the Cayley graph $X(\Gamma, \mathcal{G})$ has naturally the structure of a \emph{complex of prisms}. We begin by giving a few definitions:

\begin{definition}
Let $X$ be a graph. A \emph{clique} of $X$ is a maximal complete subgraph. A \emph{prism} $P \subset X$ is an induced subgraph which decomposes as a product of cliques of $X$ in the following sense: There exist cliques $C_1, \ldots, C_k \subset P$ of $X$ and a bijection between the vertices of $P$ and the $k$-tuples of vertices $C_1 \times \cdots \times C_k$ such that two vertices of $P$ are linked by an edge if and only if the two corresponding $k$-tuples differ on a single coordinate. The number of factors, namely $k$, is referred to as the \emph{cubical dimension} of $P$. More generally, the \emph{cubical dimension} of $X$ is the highest cubical dimension of its prisms.
\end{definition}

\noindent
The first observation is that cliques of $X(\Gamma, \mathcal{G})$ correspond to cosets of vertex groups.

\begin{lemma}\label{lem:CliqueStab}
The cliques of $X(\Gamma, \mathcal{G})$ coincide with the cosets $gG_u$, where $g \in \Gamma \mathcal{G}$ and $u \in V(\Gamma)$.
\end{lemma}

\begin{proof}
First of all, observe that the edges of a triangle of $X(\Gamma, \mathcal{G})$ are labelled by elements of $\Gamma \mathcal{G}$ that  belong to the same vertex group. Indeed, if the vertices $x,y,z \in X(\Gamma,\mathcal{G})$ generate a triangle, then $z^{-1}x$, $z^{-1}y$ and $y^{-1}x$ are three non-trivial elements of vertex groups such that $z^{-1}x = (z^{-1}y) \cdot (y^{-1}x)$. Of course, the product $(z^{-1}y) \cdot (y^{-1}x)$ cannot be reduced, which implies that $(z^{-1}y)$ and $(y^{-1}x)$ belong to the same vertex group, say $G_u$. From the previous equality, it follows that $z^{-1}x$ belongs to $G_u$ as well, concluding the proof of our claim. We record its statement for future use:

\begin{fact}\label{fact:Triangle}
In $X(\Gamma, \mathcal{G})$, the edges of a triangle are labelled by elements of a common vertex group.
\end{fact}

\noindent
As a consequence, the edges of any complete subgraph of $X(\Gamma, \mathcal{G})$ are all labelled by elements of the same vertex group. Thus, we have proved that any clique of $X(\Gamma, \mathcal{G})$ is generated by $gG_u$ for some $g \in \Gamma \mathcal{G}$ and $u \in V(\Gamma)$. Conversely, fix some $g \in \Gamma \mathcal{G}$ and $u \in V(\Gamma)$. By definition of $X(\Gamma, \mathcal{G})$, it is clearly a complete subgraph, and, if $C$ denotes a clique containing $gG_u$, we already know that $C=hG_v$ for some $h \in \Gamma \mathcal{G}$ and $v \in V(\Gamma)$. Since
$$\langle G_u,G_v \rangle = \left\{ \begin{array}{cl} G_u \times G_v & \text{if $u$ and $v$ are adjacent in $\Gamma$} \\ G_u \ast G_u & \text{if $u$ and $v$ are non-adjacent and distinct} \\ G_u=G_v & \text{if $u=v$} \end{array} \right.,$$
it follows from the inclusion $gG_u \subset C = hG_v$ that $u=v$. Finally, since two cosets of the same subgroup either coincide or are disjoint, we conclude that $gG_u=C$ is a clique of $X(\Gamma, \mathcal{G})$. 
\end{proof}

\noindent
Next, we observe that prisms of $X(\Gamma, \mathcal{G})$ correspond to cosets of subgroups generated by complete subgraphs of $\Gamma$.

\begin{lemma}\label{lem:PrismGP}
The prisms of $X(\Gamma, \mathcal{G})$ coincide with the cosets $g \langle \Lambda \rangle$ where $g \in \Gamma \mathcal{G}$ and where $\Lambda \subset \Gamma$ is a complete subgraph.
\end{lemma}

\begin{proof}
If $g \in \Gamma \mathcal{G}$ and if $\Lambda \subset \Gamma$ is a complete subgraph, then $g \langle \Lambda \rangle$ is the product of the cliques $gG_u$ where $u \in \Lambda$. A fortiori, $g \langle\Lambda \rangle$ is a prism. Conversely, let $P$ be a prism of $X(\Gamma, \mathcal{G})$. Fix a vertex $g \in P$ and let $\mathcal{C}$ be a collection of cliques all containing $g$ such that $P$ is the product of the cliques of $\mathcal{C}$. As a consequence of Lemma \ref{lem:CliqueStab}, there exists a subgraph $\Lambda \subset \Gamma$ such that $\mathcal{C}= \{ g G_u \mid u\in \Lambda \}$. Fix two distinct vertices $u,v \in \Lambda$ and two elements $a \in G_u$, $b\in G_v$. Because $P$ is a prism, the edges $(g,ga)$ and $(g,gb)$ generate a square in $X(\Gamma, \mathcal{G})$. Let $x$ denote its fourth vertex. It follows from Proposition \ref{prop:distinGP} that $g^{-1}x$ has length two and that the geodesics from $g$ to $x$ are labelled by the reduced words representing $g^{-1}x$. As $g$ and $x$ are opposite vertices in a square, there exist two geodesics between them. The only possibility is that $g^{-1}x=ab$ and that $a$ and $b$ belong to adjacent vertex groups, so that $g$, $ga$, $gb$ and $gab=gba$ are the vertices of our square. A fortiori, $u$ and $v$ are adjacent in $\Gamma$. The following is a consequence of our argument, which we record for future use:

\begin{fact}\label{fact:Square}
Two edges of $X(\Gamma,\mathcal{G})$ sharing an endpoint generate a square if and only if they are labelled by adjacent vertex groups. If so, two opposite sides of the square are labelled by the same element of $\Gamma \mathcal{G}$.
\end{fact}

\noindent
Thus, we have proved that $\Lambda$ is a complete subgraph of $\Gamma$. Since the prisms $P$ and $g \langle \Lambda \rangle$ both coincide with the product of the cliques of $\mathcal{C}$, we conclude that $P=g \langle \Lambda \rangle$, proving our lemma.
\end{proof}

\noindent
An immediate consequence of Lemma \ref{lem:PrismGP} is the following statement:

\begin{cor}\label{cor:cubdim}
 The cubical dimension of $X(\Gamma, \mathcal{G})$ is equal to $\mathrm{clique}(\Gamma)$, the maximal cardinality of a complete subgraph of $\Gamma$.
\end{cor}

\paragraph{Hyperplanes.} Now, our goal is to define \emph{hyperplanes} in $X(\Gamma, \mathcal{G})$ and to show that they behave essentially in the same way as hyperplanes in CAT(0) cube complexes. 

\begin{definition}\label{def:hyp}
A \emph{hyperplane} of $X(\Gamma, \cG)$ is a class of edges with respect to the transitive closure of the relation claiming that two edges which are opposite in a square or which belong to a common triangle are equivalent. The \emph{carrier} of a hyperplane $J$, denoted by $N(J)$, is the subgraph of $X(\Gamma, \cG)$ generated by $J$. Two hyperplanes $J_1$ and $J_2$ are \emph{transverse} if they intersect a square along two distinct pairs of opposite edges, and they are \emph{tangent} if they are distinct, not transverse and if their carriers intersect.
\end{definition}
\begin{figure}[H]
\begin{center}
\includegraphics[trim={0 16.5cm 10cm 0},clip,scale=0.6]{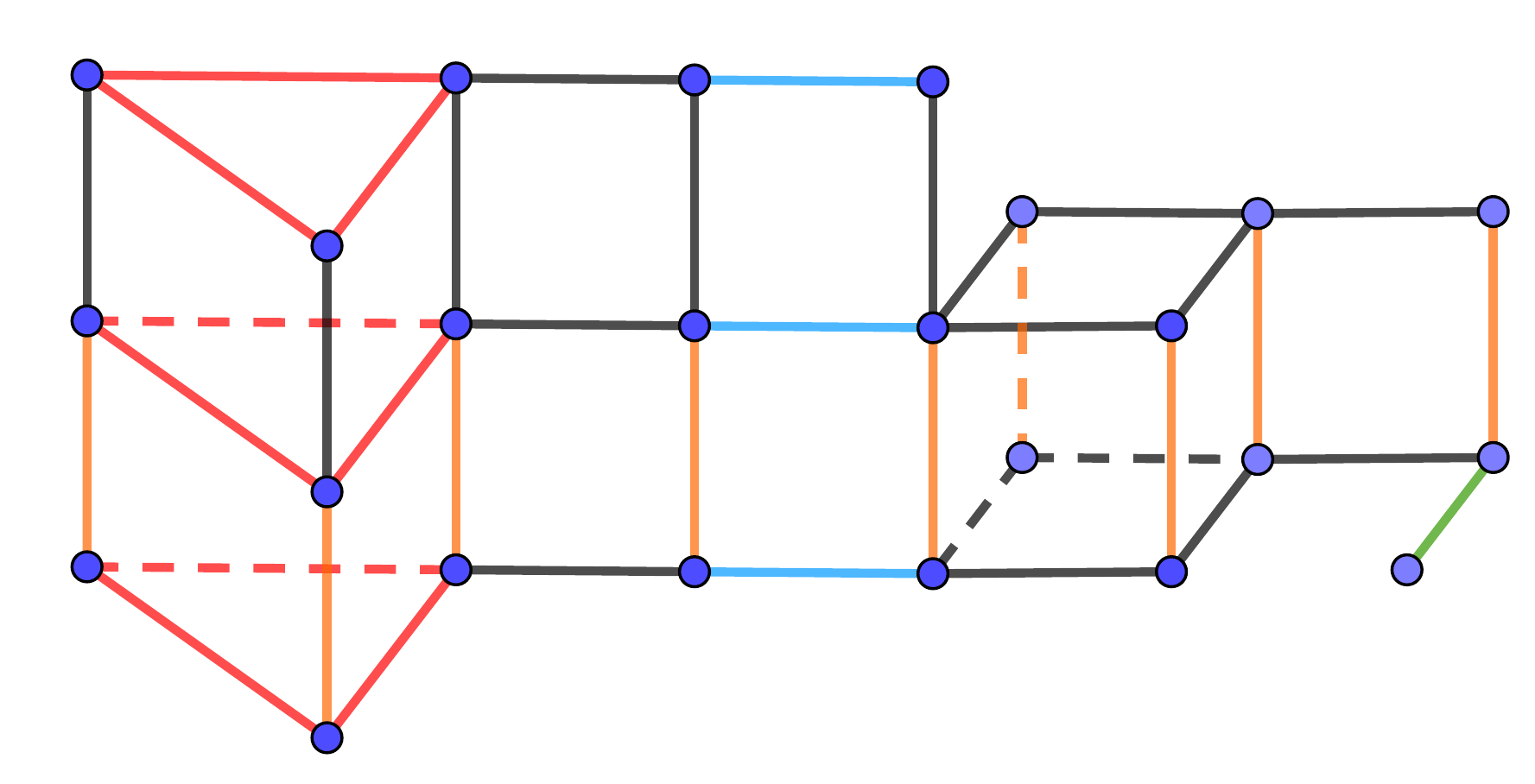}
\caption{Four hyperplanes in a quasi-median graph, colored in red, blue, green, and orange. The orange hyperplane is transverse to the blue and red hyperplanes. The green and orange hyperplanes are tangent.}
\label{figure3}
\end{center}
\end{figure}

\noindent
We refer to Figure \ref{figure3} for examples of hyperplanes in a graph. We begin by describing the hyperplanes of $X(\Gamma, \mathcal{G})$. For convenience, for every vertex $u \in V(\Gamma)$ we denote by $J_u$ the hyperplane which contains all the edges of the clique $G_u$. Our description of the hyperplanes of $X(\Gamma, \mathcal{G})$ is the following:

\begin{thm}\label{thm:HypStab}
For every hyperplane $J$ of $X(\Gamma, \mathcal{G})$, there exist some $g \in \Gamma \mathcal{G}$ and $u \in V(\Gamma)$ such that $J=gJ_u$. Moreover, $N(J)= g \langle \mathrm{star}(u) \rangle$ and $\mathrm{stab}(J)= g \langle \mathrm{star}(u) \rangle g^{-1}$. 
\end{thm}

\noindent
The key step in proving Theorem \ref{thm:HypStab} is the following characterisation:

\begin{prop}\label{prop:EdgesDualHyp}
Fix a vertex $u \in V(\Gamma)$ and two adjacent vertices $x,y \in X(\Gamma ,\mathcal{G})$. The following statements are equivalent:
\begin{itemize}
	\item[(i)] the edge $(x,y)$ is dual to the hyperplane $J_u$;
	\item[(ii)] $x \in \langle \mathrm{star}(u) \rangle$ and $x^{-1}y \in G_u$;
	\item[(iii)] the projections of $x$ and $y$ onto the clique $G_u$ are distinct.
\end{itemize}
\end{prop}

\noindent
The third point requires an explanation. The \emph{projection} of a vertex onto a clique refers to the unique vertex of the clique which minimises the distance to our initial vertex. The existence of such a projection is justified by the following lemma:

\begin{lemma}\label{lem:ProjClique}
Fix a vertex $u \in V(\Gamma)$ and let $g \in X(\Gamma, \mathcal{G})$. There exists a unique vertex of the clique $G_u$ minimising the distance to $g$, namely $\left\{ \begin{array}{cl} 1 & \text{if $\mathrm{head}(g) \cap G_u = \emptyset$} \\ \mathrm{head}(g) \cap G_u & \text{if $\mathrm{head}(g) \cap G_u \neq \emptyset$} \end{array} \right.$.
\end{lemma}

\begin{proof}
Suppose that $\mathrm{head}(g) \cap G_u = \emptyset$. Then, for every $h \in G_u$, one has 
$$d(g,h)=|h^{-1}g|= |g| +1 =d(g,1)+1>d(g,1)$$ 
since the product $h^{-1}g$ is necessarily reduced. It shows that $1$ is the unique vertex of $G_u$ minimising the distance to $g$. Next, suppose that $\mathrm{head}(g) \cap G_u \neq \emptyset$. Thus,  we can write $g$ as a reduced product $hg'$ where $h$ belongs to $G_u \backslash \{ 1 \}$ and where $g' \in \Gamma \mathcal{G}$ satisfies $\mathrm{head}(g') \cap G_u = \emptyset$. Notice that $\mathrm{head}(g) \cap G_u = \{h\}$. Then 
$$d(g,k) = |k^{-1}g| = |(k^{-1}h) \cdot g'| = |g'| +1 = d(g,h) +1>d(g,h)$$
for every $k \in G_u \backslash \{h \}$ since the product $(k^{-1}h) \cdot g'$ is necessarily reduced. This proves that $h$ is the unique vertex of $G_u$ minimising the distance to $g$.
\end{proof}

\begin{proof}[Proof of Proposition \ref{prop:EdgesDualHyp}.]
Suppose that $(i)$ holds. There exists a sequence of edges
$$(x_1,y_1), \ (x_2,y_2), \ldots, (x_{n-1},y_{n-1}), \ (x_n,y_n)=(x,y)$$
such that $(x_1,y_1) \subset G_u$, and such that, for every $1 \leq i \leq n-1$, the edges $(x_i,y_i)$ and $(x_{i+1},y_{i+1})$ either belong to the same triangle or are opposite in a square. We argue by induction over $n$. If $n=1$, there is nothing to prove. Now suppose that $x_{n-1} \in \langle \mathrm{star}(u) \rangle$ and that $x_{n-1}^{-1}y_{n-1} \in G_u$. If $(x_{n-1},y_{n-1})$ and $(x_n,y_n)$ belong to the same triangle, it follows from Fact \ref{fact:Triangle} that $x_n \in \langle \mathrm{star}(u) \rangle$ and $x_n^{-1}y_n \in G_u$. Otherwise, if $(x_{n-1},y_{n-1})$ and $(x_n,y_n)$ are opposite sides in a square, we deduce from Fact \ref{fact:Triangle} that there exists some $a \in \langle \mathrm{link}(u) \rangle$ such that either $\left\{ \begin{array}{l} x_n=x_{n-1}a \\ y_n = y_{n-1}a \end{array} \right.$ or $\left\{ \begin{array}{l} x_n=y_{n-1}a \\ y_n = x_{n-1}a \end{array} \right.$. As a consequence, $x_n \in \langle \mathrm{star}(u) \rangle$ and $x_n^{-1}y_n \in G_u$. Thus, we have proved the implication $(i) \Rightarrow (ii)$. 

\medskip \noindent
Now, suppose that $(ii)$ holds. There exists some $\ell \in G_u \backslash \{ 1 \}$ such that $y=x \ell$, and, since $\langle \mathrm{star}(u) \rangle = G_u\times \langle \mathrm{link}(u) \rangle$, we can write $x$ as a reduced product $ab$ for some $a \in G_u$ and $b \in \langle \mathrm{link}(u) \rangle$. Notice that $y$ is represented by the reduced product $(a \ell) \cdot b$. We deduce from Lemma \ref{lem:ProjClique} that the projections of $x$ and $y$ onto the clique $G_u$ are $a$ and $a \ell$ respectively. They are distinct since $\ell \neq 1$. Thus, we have proved the implication $(ii) \Rightarrow (iii)$.

\medskip \noindent
Suppose that $(iii)$ holds. Write $x$ as a product $a \cdot x_1 \cdots x_n$, where $a \in G_u$ and $x_1, \ldots, x_n \in \Gamma \mathcal{G}$ are generators, such that $a=1$ and $x_1 \cdots x_n$ is reduced if $\mathrm{head}(g) \cap G_u = \emptyset$, and such that $a \cdot x_1 \cdots x_n$ is reduced otherwise; notice that in the latter case, $\mathrm{head}(g) \cap G_u = \{a \}$. According to Lemma \ref{lem:ProjClique}, the projection of $x$ onto the clique $G_u$ is $a$. Next, because $x$ and $b$ are adjacent, there exists a generator $b \in \Gamma \mathcal{G}$ such that $y = xb$. Since $x$ and $y$ must have different projections onto the clique $G_u$, we deduce from Lemma \ref{lem:ProjClique} that necessarily $b$ shuffles to the left in the product $x_1 \cdots x_n \cdot b$ (see Section \ref{section:GP} for the definition) and belongs to $G_u$. (In this case, the projection of $y$ onto the clique $G_u$ is $ab$, which distinct from $a$ since $b \neq 1$.) As a consequence, the $x_i$'s belong to $\langle \mathrm{link}(u) \rangle$. Finally, it is sufficient to notice that any two consecutive edges of the sequence
$$(a,ab), \ (ax_1,abx_1),  \ldots, (ax_1 \cdots x_{n-1},abx_1\cdots x_{n-1}), \ (ax_1 \cdots x_n, abx_1 \cdots x_n)= (x,y)$$
are opposite sides of a square in order to deduce that $(x,y)$ and $(a,ab) \subset G_u$ are dual to the same hyperplane, namely $J_u$. Thus, we have proved the implication $(iii) \Rightarrow (i)$.  
\end{proof}

\begin{proof}[Proof of Theorem \ref{thm:HypStab}.]
Let $J$ be a hyperplane of $X(\Gamma, \mathcal{G})$. Fixing a clique $C$ dual to $J$, we know from Lemma \ref{lem:CliqueStab} that there exist $g \in \Gamma \mathcal{G}$ and $u \in V(\Gamma)$ such that $C=gG_u$, hence $J=gJ_u$. It is a consequence of Proposition \ref{prop:EdgesDualHyp} that a vertex of $X(\Gamma,\mathcal{G})$ belongs to $N(J_u)$ if and only if it belongs to $\langle \mathrm{star}(u) \rangle$, so $N(J)=gN(J_u)=g \langle \mathrm{star}(u) \rangle$. 

\medskip \noindent
It remains to show that $\mathrm{stab}(J)= g \langle \mathrm{star}(u) \rangle g^{-1}$. Fix a non-trivial element $a \in G_u$. Then the hyperplane dual to the edge $(g,ga)$ is $J$. If $h \in \mathrm{stab}(J)$, then $J$ must be also dual to $h \cdot (g,ga)$, and we deduce from Proposition \ref{prop:EdgesDualHyp} that $hg$ must belong to $g \langle \mathrm{star}(u)$, hence $h \in g \langle \mathrm{star}(u) \rangle g^{-1}$. Conversely, if $h$ belongs to $g \langle \mathrm{star}(u) \rangle g^{-1}$, then $hg \in g \langle \mathrm{star}(u) \rangle$ and $(hg)^{-1}(ga) \in G_u$ so that Proposition \ref{prop:EdgesDualHyp} implies that $gJ_u=J$ is the hyperplane dual to the edge $h \cdot (g,ga)$, hence $hJ=J$ since these two hyperplanes turn out to be dual to the same edge. This concludes the proof of the theorem.
\end{proof}

\noindent
It is worth noticing that, as a consequence of Theorem \ref{thm:HypStab}, the hyperplanes of $X(\Gamma, \mathcal{G})$ are naturally labelled by $V(\Gamma)$. More precisely, since any hyperplane $J$ of $X(\Gamma, \mathcal{G})$ is a translate of some $J_u$, we say that the corresponding vertex $u \in V(\Gamma)$ \emph{labels} $J$. Equivalently, by noticing that the edges of $X(\Gamma, \mathcal{G})$ are naturally labelled by vertices of $\Gamma$, the vertex of $\Gamma$ labelling a hyperplane coincides with the common label of all its edges (as justified by Facts \ref{fact:Triangle} and \ref{fact:Square}). Let us record the following elementary but quite useful statement:

\begin{lemma}\label{lem:transverseimpliesadj}
Two transverse hyperplanes of $X(\Gamma, \mathcal{G})$ are labelled by adjacent vertices of $\Gamma$, and two tangent hyperplanes of $X(\Gamma, \mathcal{G})$ are labelled by distinct vertices of $\Gamma$. 
\end{lemma}

\begin{proof}
The assertion about transverse hyperplanes is a direct consequence of Fact \ref{fact:Square}. Now, let $J_1$ and $J_2$ be two tangent hyperplanes, and let $u_1,u_2 \in V(\Gamma)$ denote their labels respectively. Since these two hyperplanes are tangent, there exists a vertex $g \in X(\Gamma, \mathcal{G})$ which belongs to both $N(J_1)$ and $N(J_2)$. Fix two cliques, say $C_1$ and $C_2$ respectively, containing $g$ and dual to $J_1$ and $J_2$. According to Lemma \ref{lem:CliqueStab}, we have $C_1=gG_{u_1}$ and $C_2=gG_{u_2}$. Clearly, $u_1$ and $u_2$ must be distinct, since otherwise $C_1$ and $C_2$ would coincide, contradicting the fact that $J_1$ and $J_2$ are tangent. Therefore, our two hyperplanes $J_1$ and $J_2$ are indeed labelled by distinct vertices of $\Gamma$. 
\end{proof}

\noindent
Now we want to focus on the second goal of this paragraph by showing that hyperplanes of $X(\Gamma, \mathcal{G})$ are closely related to its geometry. Our main result is the following:

\begin{thm}\label{thm:QmHyperplanes}
The following statements hold:
\begin{itemize}
	\item For every hyperplane $J$, the graph $X(\Gamma, \mathcal{G}) \backslash \backslash J$ is disconnected. Its connected components are called \emph{sectors}.
	\item Carriers of hyperplanes are convex.
	\item For any two vertices $x,y \in X(\Gamma, \mathcal{G})$, $d(x,y)= \# \{ \text{hyperplanes separating $x$ and $y$} \}$.
	\item A path in $X(\Gamma, \mathcal{G})$ is a geodesic if and only if it intersects each hyperplane at most once.
\end{itemize}
\end{thm}

\noindent
In this statement, we denoted by $X(\Gamma, \mathcal{G}) \backslash \backslash J$, where $J$ is a hyperplane, the graph obtained from $X(\Gamma, \mathcal{G})$ by removing the interiors of the edges of $J$.

\begin{proof}[Proof of Theorem \ref{thm:QmHyperplanes}.]
Let $J$ be a hyperplane of $X(\Gamma, \mathcal{G})$. Up to translating $J$ by an element of $\Gamma \mathcal{G}$, we may suppose without loss of generality that $J=J_u$ for some $u \in V(\Gamma)$. Fix a non-trivial element $a \in G_u$. We claim that the vertices $1$ and $a$ are separated by $J_u$. Indeed, if $x_1, \ldots, x_n$ define a path from $1$ to $a$ in $X(\Gamma,\mathcal{G})$, there must exist some $1 \leq i \leq n-1$ such that the projections of $x_i$ and $x_{i+1}$ onto the clique $G_u$ are distinct, since the projections of $1$ and $a$ onto $G_u$ are obviously $1$ and $a$ respectively and are distinct. It follows from Proposition \ref{prop:EdgesDualHyp} that the edge $(x_i,x_{i+1})$ is dual to $J_u$. This concludes the proof of the first point in the statement of our theorem.

\medskip \noindent
The convexity of carriers of hyperplanes is a consequence of the characterisation of geodesics given by Proposition \ref{prop:distinGP} and of the description of carriers given by Theorem~\ref{thm:HypStab}. 

\medskip \noindent
Let $\gamma$ be a geodesic between two vertices $x,y \in X(\Gamma, \mathcal{G})$. Suppose by contradiction that it intersects a hyperplane at least twice. Let $e_1, \ldots, e_n$ be the sequence of edges corresponding to $\gamma$. Fix two indices $1 \leq i <j \leq n$ such that $e_i$ and $e_j$ are dual to the same hyperplane, say $J$, and such that the subpath $\rho = e_{i+1} \cup \cdots \cup e_{j-1}$ intersects each hyperplane at most once and does not intersect $J$. Notice that, as a consequence of the convexity of $N(J)$, this subpath must be contained in the carrier $N(J)$. Therefore, any hyperplane dual to an edge of $\rho$ must be transverse to $J$. It follows that, if $x_i$ denotes the generator labelling the edge $e_i$ for every $1 \leq i \leq n$, then $x_i$ shuffles to the end in the product $x_i \cdot x_{i+1} \cdots x_{j-1}$; moreover, $x_i$ and $x_j$ belong to the same vertex group. Consequently, the product $x_1 \cdots x_n$ is not reduced, contradicting the fact that $\gamma$ is a geodesic according to Proposition \ref{prop:distinGP}. Thus, we have proved that a geodesic in $X(\Gamma, \mathcal{G})$ intersects each hyperplane at most once. This implies the inequality
$$d(x,y) \leq \# \{ \text{hyperplanes separating $x$ and $y$}\}.$$
The reverse inequality is clear since any path from $x$ to $y$ must intersect each hyperplane separating $x$ and $y$, proving the equality. As a consequence, if a path between $x$ and $y$ intersects each hyperplane at most once, then its length coincides with the number of hyperplanes separating $x$ and $y$, which coincides itself with the distance between $x$ and $y$. A fortiori, such a path must be a geodesic. This concludes the proof of the third and fourth points in the statement of the theorem. 
\end{proof}

\paragraph{Projections on parabolic subgroups.} We saw in Lemma \ref{lem:ProjClique} that it is possible to project naturally vertices of $X(\Gamma, \mathcal{G})$ onto a given clique. Now, we want to extend this observation to a wider class of subgraphs. More precisely, if $\Lambda$ is a subgraph of $\Gamma$, then we claim that vertices of $X(\Gamma, \mathcal{G})$ project onto the subgraph $\langle \Lambda \rangle$. This covers cliques but also carriers of hyperplanes according to Theorem \ref{thm:HypStab}. Before stating and proving the main result of this paragraph, we would like to emphasize that, as a consequence of Proposition \ref{prop:distinGP}, such a subgraph $\langle \Lambda \rangle$ is necessarily convex.

\begin{prop}\label{prop:ProjHyp}
Fix a subgraph $\Lambda \subset \Gamma$ and a vertex $g \in X(\Gamma, \mathcal{G})$. There exists a unique vertex $x$ of $\langle \Lambda \rangle$ minimising the distance to $g$. Moreover, any hyperplane separating $g$ from $x$ separates $g$ from $\langle \Lambda \rangle$ (ie., $g$ and $\langle \Lambda \rangle$ lie in distinct sectors delimited by the hyperplane). 
\end{prop}

\begin{proof}
Fix a vertex $x \in \langle \Lambda \rangle$ minimising the distance to $g$, and a geodesic $[g,x]$ from $g$ to $x$. We say that an edge of $[g,x]$ is \emph{bad} if the hyperplane dual to it crosses $\langle \Lambda \rangle$. Let $e$ be the bad edge of $[g,x]$ which is closest to $x$. As a consequence, the edges of $[g,x]$ between $e$ and $x$ have their hyperplanes which are disjoint from $\langle \Lambda \rangle$. This implies that these hyperplanes are all transverse to the hyperplane $J$ dual to $e$, so that, as a consequence of Lemma \ref{lem:transverseimpliesadj}, the syllable $s$ of $g^{-1}x$ labelling $e$ belongs to the tail of $g^{-1}x$. Moreover, the fact that $J$ crosses $\langle \Lambda \rangle$ implies that it is labelled by a vertex of $\Lambda$, hence $s \in \langle \Lambda \rangle$. We deduce from Proposition \ref{prop:distinGP} that there exists a geodesic from $g$ to $x$ whose last edge is labelled by $s$. Since $s$ and $x$ both belong to $\langle \Lambda \rangle$, it follows that the penultimate vertex along our geodesic, namely $xs^{-1}$, belongs to $\langle \Lambda \rangle$ and satisfies $d(g,xs^{-1}) <d(g,x)$, contradicting the definition of $x$. Thus, we have proved that a geodesic from $g$ to $x$ does not contain any bad edge. In other words, any hyperplane separating $g$ from $x$ separates $g$ from~$\langle \Lambda \rangle$.  This proves the second assertion of our proposition. 

\medskip \noindent
Now, suppose that $y \in \langle \Lambda \rangle$ is a second vertex minimising the distance to $g$. If $x$ and $y$ are distinct, then there exists a hyperplane $J$ separating them. Because such a hyperplane necessarily crosses $\langle \Lambda \rangle$, we deduce from the first paragraph of our proof that $J$ does not separate $g$ from $x$; similarly, $J$ does not separate $g$ from $y$. But this implies that $J$ does not separate $x$ and $y$, contradicting the choice of $J$. This proves that $x$ and $y$ necessarily coincide, concluding the proof of our proposition.
\end{proof}

\noindent
Below, we record several easy consequences of Proposition \ref{prop:ProjHyp}.

\begin{cor}\label{cor:hypsepprojections}
Let  $\Lambda$ be a subgraph of $\Gamma$ and let $x,y \in X(\Gamma, \mathcal{G})$ be two vertices. The hyperplanes separating the projections of $x$ and $y$ onto $\langle \Lambda \rangle$ are precisely the hyperplanes separating $x$ and $y$ which intersect $\langle \Lambda \rangle$. In particular, any hyperplane separating these projections also separates $x$ and $y$. 
\end{cor}

\begin{proof}
Let $x',y' \in \langle \Lambda \rangle$ denote respectively the projections of $x$ and $y$ onto $\langle \Lambda \rangle$. If $J$ is a hyperplane separating $x'$ and $y'$ then it has to cross $\langle \Lambda \rangle$. As a consequence of Proposition \ref{prop:ProjHyp}, $J$ cannot separate $x$ and $x'$ nor $y$ and $y'$. Therefore, it has to separate $x$ and $y$. Conversely, suppose that $J$ is a hyperplane separating $x$ and $y$ which intersects $\langle \Lambda \rangle$. Once again according to Proposition \ref{prop:ProjHyp}, $J$ cannot separate $x$ and $x'$ nor $y$ and $y'$. Therefore, it has to separate $x'$ and $y'$. This concludes the proof of our lemma. 
\end{proof}

\begin{cor}\label{cor:diamproj}
Let  $\Lambda, \Xi \subset \Gamma$ be two subgraphs and let $g,h \in \Gamma \mathcal{G}$. The diameter of the projection of $g \langle \Lambda \rangle$ onto $h \langle \Xi \rangle$ is at most the number of hyperplanes intersecting both $g\langle \Lambda \rangle$ and $h\langle \Xi \rangle$.
\end{cor}

\begin{proof}
For convenience, let $p : X(\Gamma, \mathcal{G}) \to h \langle \Xi \rangle$ denote the projection onto $h \langle \Xi \rangle$. Let $D$ denote the number (possibly infinite) of hyperplanes intersecting both $g\langle \Lambda \rangle$ and $h\langle \Xi \rangle$. We claim that, for every vertices $x,y \in g \langle \Lambda \rangle$, the distance between $p(x)$ and $p(y)$ is at most $D$. Indeed, as a consequence of Corollary \ref{cor:hypsepprojections}, any hyperplane separating $p(x)$ and $p(y)$ separates $x$ and $y$, so that any hyperplane separating $p(x)$ and $p(y)$ must intersect both $g\langle \Lambda \rangle$ and $h\langle \Xi \rangle$. Consequently, the diameter of $p(g \langle \Lambda \rangle)$ is at most $D$.
\end{proof}

\begin{cor}\label{cor:minseppairhyp}
Let  $\Lambda, \Xi \subset \Gamma$ be two subgraphs and let $g,h \in \Gamma \mathcal{G}$ be two elements. Fix two vertices $x \in g \langle \Lambda \rangle$ and $y \in h \langle \Xi \rangle$ minimising the distance between $g\langle \Lambda \rangle$ and $h\langle \Xi \rangle$. The hyperplanes separating $x$ and $y$ are precisely those separating $g\langle \Lambda \rangle$ and $h\langle \Xi \rangle$. 
\end{cor}

\begin{proof}
Let $J$ be a hyperplane separating $x$ and $y$. Notice that $x$ is the projection of $y$ onto $g\langle \Lambda \rangle$, and similarly $y$ is the projection of $x$ onto $h\langle \Xi \rangle$. By applying Proposition \ref{prop:ProjHyp} twice, it follows that $J$ is disjoint from both $g\langle \Lambda \rangle$ and $h\langle \Xi \rangle$. Consequently, $J$ separates $g\langle \Lambda \rangle$ and $h\langle \Xi \rangle$. Conversely, it is clear that any hyperplane separating $g\langle \Lambda \rangle$ and $h\langle \Xi \rangle$ also separates $x$ and $y$.
\end{proof}

\begin{cor}
Let  $\Lambda, \Xi \subset \Gamma$ be two subgraphs and let $g,h \in \Gamma \mathcal{G}$ be two elements. If $g \langle \Lambda \rangle \cap h \langle \Xi \rangle = \emptyset$ then there exists a hyperplane separating $g\langle \Lambda \rangle$ and $h\langle \Xi \rangle$. 
\end{cor}

\begin{proof}
Fix two vertices $x \in g \langle \Lambda \rangle$ and $y \in h \langle \Xi \rangle$ minimising the distance between $g\langle \Lambda \rangle$ and $h\langle \Xi \rangle$. Because these two subgraphs are disjoint, $x$ and $y$ must be distinct. According to Corollary \ref{cor:minseppairhyp}, taking a hyperplane separating $x$ and $y$ provides the desired hyperplane.
\end{proof}

\paragraph{Hyperplane stabilisers.}

\noindent
A useful tool when working with the Cayley graph $X(\Gamma, \mathcal{G})$ is the notion of \emph{rotative-stabiliser}.

\begin{definition}
Let $\Gamma$ be a simplicial graph and $\mathcal{G}$ a collection of groups indexed by $V(\Gamma)$. Given a hyperplane $J$ of $X(\Gamma, \mathcal{G})$, its \emph{rotative-stabiliser} is the following subgroup of $\Gamma \mathcal{G}$:
$$\mathrm{stab}_{\circlearrowleft}(J) := \bigcap\limits_{\text{$C$ clique dual to $J$}} \mathrm{stab}(C).$$
\end{definition}

\noindent
We begin by describing rotative-stabilisers of hyperplanes in $X(\Gamma, \mathcal{G})$. More precisely, our first main result is the following:

\begin{prop}\label{prop:rotstabinX}
The rotative-stabiliser of a hyperplane $J$ of $X(\Gamma, \mathcal{G})$ coincides with the stabiliser of any clique dual to $J$. Moreover, $\mathrm{stab}_{\circlearrowleft}(J)$ acts freely and transitively on the set of sectors delimited by $J$, and it stabilises each sector delimited by the hyperplanes transverse to $J$; in particular, it stabilises the hyperplanes transverse to $J$. 
\end{prop}

\begin{proof}
Let $J$ be a hyperplane of $X(\Gamma, \mathcal{G})$. Up to translating $J$ by some element of $\Gamma \mathcal{G}$, we may suppose without loss of generality that $J=J_u$ for some $u \in V(\Gamma)$. As a consequence of Proposition \ref{prop:EdgesDualHyp}, the cliques of $X(\Gamma, \mathcal{G})$ dual to $J_u$ correspond to the cosets $g G_u$ where $g \in \langle \mathrm{link}(u) \rangle$. Clearly, they all have the same stabiliser, namely $G_u$. This proves the first assertion of our proposition.

\medskip \noindent
Next, if follows from Proposition \ref{prop:EdgesDualHyp} that two vertices of $X(\Gamma, \mathcal{G})$ belong to the same sector delimited by $J_u$ if and only if they have the same projection onto the clique $G_u$. Therefore, the collection of sectors delimited by $J_u$ is naturally in bijection with the vertices of the clique $G_u$. Since $\mathrm{stab}_\circlearrowleft(J_u)=G_u$ acts freely and transitively on the vertices of the clique $G_u$, it follows that this rotative-stabiliser acts freely and transitively on the set of sectors delimited by $G_u$. 

\medskip \noindent
Finally, let $J_1$ and $J_2$ be two transverse hyperplanes. Up to translating $J_1$ and $J_2$ by an element of $\Gamma \mathcal{G}$, we may suppose without loss of generality that the vertex $1$ belongs to $N(J_1) \cap N(J_2)$. As a consequence, there exist vertices $u,v \in V(\Gamma)$ such that $J_1=J_u$ and $J_2=J_v$. According to Lemma \ref{lem:transverseimpliesadj}, $u$ and $v$ are adjacent in $\Gamma$. As a by-product, one gets the following statement, which we record for future use:

\begin{fact}\label{fact:RotStabCom}
The rotative-stabilisers of two transverse hyperplanes of $X(\Gamma, \mathcal{G})$ commute, ie., any element of one rotative-stabiliser commutes with any element of the other.
\end{fact}

\noindent 
For every vertex $x \in X(\Gamma, \mathcal{G})$ and every element $g \in \mathrm{stab}_\circlearrowleft(J_u)=G_u$, we deduce from Lemma \ref{lem:ProjClique} that $x$ and $gx$ have the same projection onto the clique $G_v$ since the vertex groups $G_u$ and $G_v$ commute. Because two vertices of $X(\Gamma, \mathcal{G})$ belong to the same sector delimited by $J_v$ if and only if they have the same projection onto the clique $G_v$, according to Proposition~\ref{prop:EdgesDualHyp}, we conclude that $\mathrm{stab}_\circlearrowleft(J_u)$ stabilises each sector delimited by $J_v$.
\end{proof}

\noindent
We also record the following preliminary lemma which will be used later.

\begin{lemma}\label{lem:shortendist}
Let $x \in X(\Gamma, \mathcal{G})$ be a vertex and let $J,H$ be two hyperplanes of $X(\Gamma, \mathcal{G})$. Suppose that $J$ separates $x$ from $H$ and let $g \in \mathrm{stab}_\circlearrowleft(J)$ denote the unique element sending $H$ into the sector delimited by $J$ which contains $x$. Then $d(1,N(gH))<d(1,N(H))$.
\end{lemma}

\begin{proof}
Let $y \in N(H)$ denote the projection of $x$ onto $N(H)$ and fix a geodesic $[x,y]$ between $x$ and $y$. Because $J$ separates $x$ from $H$, $[x,y]$ must contain an edge $[a,b]$ dual to $J$. Let $[x,a]$ and $[b,y]$ denote the subpath of $[x,y]$ between $x$ and $a$, and $b$ and $y$, respectively. Notice that $gb=a$ since $g$ stabilises the clique containing $[a,b]$ and sends the sector delimited by $J$ which contains $H$ (and a fortiori $b$) to the sector delimited by $J$ which contains $x$ (and a fortiori $a$). As a consequence, $[x,a] \cup g [b,y]$ defines a path from $x$ to $gN(H)$ of length $d(x,y)-1$, so that 
$$d(1,gN(H)) \leq d(x,y)-1 = d(1,N(H))-1,$$
concluding the proof.
\end{proof}

\noindent
One feature of rotative-stabilisers is that they can be used to play ping-pong. As an illustration, we prove a result that will be fundamental in Section~\ref{section:ConjAut}. Let us first introduce some notation:

\begin{definition}\label{def:peripheral}
 A collection of hyperplanes $\mathcal{J}$ of $X(\Gamma, \mathcal{G})$ is \emph{peripheral} if, for every $J_1,J_2 \in \mathcal{J}$, $J_1$ does not separate $1$ from $J_2$. 
\end{definition}

\begin{lemma}\label{lem:pingpong}
Fix a collection of hyperplanes $\mathcal{J}$, and, for every $J \in \mathcal{J}$, let $S(J)$ denote the sector delimited by $J$ that contains $1$. If $\mathcal{J}$ is peripheral, then $g \notin \bigcap\limits_{J \in \mathcal{J}} S(J)$ for every non-trivial $g \in \langle \mathrm{stab}_\circlearrowleft (J) \mid J \in \mathcal{J} \rangle$. 
\end{lemma}

\begin{proof}
For every $J \in \mathcal{J}$, let $R(J)$ denote the union of all the sectors delimited by $J$ that do not contain the vertex $1$. In order to prove our lemma we have  to show that $g \in \bigcup\limits_{J \in \mathcal{J}} R(J)$ for every non-trivial element $g \in \langle \mathrm{stab}_\circlearrowleft(J) \mid J \in \mathcal{J} \rangle$. We deduce from Proposition \ref{prop:rotstabinX} that:
\begin{itemize}
	\item If $J_1,J_2 \in \mathcal{J}$ are two transverse hyperplanes, then $g \cdot R(J_1) = R(J_1)$ for every $g \in \mathrm{stab}_\circlearrowleft (J_2)$;
	\item If $J_1,J_2 \in \mathcal{J}$ are two distinct hyperplanes which are not transverse, then $g \cdot R(J_1)$ is contained in $R(J_2)$ for every non-trivial $g \in \mathrm{stab}_\circlearrowleft(J_2)$.
	\item For every hyperplane $J \in \mathcal{J}$ and every non-trivial element $g \in \mathrm{stab}_\circlearrowleft (J)$, $g$ belongs to $R(J)$.
\end{itemize}
Let $\Phi$ be the graph whose vertex-set is $\mathcal{J}$ and whose edges connect two transverse hyperplanes, and set $\mathcal{H}= \{ G_J= \mathrm{stab}_\circlearrowleft(J) \mid J \in \mathcal{J} \}$. As a consequence of Fact \ref{fact:RotStabCom}, we have a natural surjective morphism $\phi : \Phi \mathcal{H} \to \langle \mathrm{stab}_\circlearrowleft(J) \mid J \in \mathcal{J} \rangle$. It follows that a non-trivial element $g \in \langle \mathrm{stab}_\circlearrowleft(J) \mid J \in \mathcal{J} \rangle$ can be represented as a non-empty and reduced word in $\Phi \mathcal{H}$, say $w$ such that $\phi(w)=g$. We claim that $g$ belongs to $R(J)$ for some vertex $J \in V(\Phi)$ such that $\mathrm{head}(w)$ contains a syllable of $G_J$. 

\medskip \noindent
We argue by induction on the length of $w$. If $w$ has length one, then $w \in G_J \backslash \{ 1 \}$ for some $J \in V(\Phi)$. Our third point above implies that $g \in R(J)$. Now, suppose that $w$ has length at least two. Write $w=ab$ where $a$ is the first syllable of $w$ and $b$ the rest of the word. Thus, $a \in G_J \backslash \{ 1 \}$ for some $J \in V(\Phi)$. We know from our induction hypothesis that $\phi(b) \in R(I)$ where $I$ is a vertex of $\Phi$ such that $\mathrm{head}(b)$ contains a syllable of $G_I$. Notice that $I \neq J$ since otherwise the word $w=ab$ would not be reduced. Two cases may happen: either $I$ and $J$ are not adjacent in $\Phi$, so that our second point above implies that $g = \phi(ab) \in \phi(a) \cdot R(I) \subset R(J)$; or $I$ and $J$ are adjacent, so that our first point above implies that $g= \phi(ab) \in \phi(a) \cdot R(I) = R(I)$. It is worth noticing that, in the former case, $\mathrm{head}(w)$ contains a syllable of $G_J$, namely $a$; in the latter case, we know that we can write $b$ as a reduced product $cd$ where $c$ is a syllable of $G_I$, hence $w=ab=acd=cad$ since $a$ and $c$ belong to the commuting vertex groups $G_I$ and $G_J$, which implies that $\mathrm{head}(w)$ contains a syllable of $G_I$. This concludes the proof of our lemma.
\end{proof}

\subsection{The Davis complex associated to a graph product of groups}\label{section:Davis}

In this section, we recall an important complex associated to a graph product of groups, whose structure will be used in Section \ref{section:atomic}.

\begin{definition}[Davis complex]\label{def:Davis}
	The \textit{Davis complex} $D(\Gamma, \cG)$ associated to the graph product $\Gamma\cG$ is defined as follows: 
	\begin{itemize}
		\item Vertices correspond to left cosets of the form $g \langle \Lambda \rangle$ for $g\in \Gamma\cG$ and $\Lambda \subset \Gamma$ a (possibly empty) complete subgraph.
		\item For every $g\in \Gamma\cG$ and complete subgraphs $\Lambda_1, \Lambda_2 \subset \Gamma$ that differ by exactly one vertex $v$, one puts an edge between the vertices $g\langle \Lambda_1\rangle$ and $g \langle \Lambda_2\rangle$. The vertex $v$ is called the \textit{label} of that edge.
		\item One obtains a cubical complex from this graph by adding for every $k \geq 2$ a $k$-cube for every  subgraph isomorphic to the $1$-skeleton of a $k$-cube.
		\end{itemize}
This complex comes with an action of $\Gamma\cG$: The group $\Gamma\cG$ acts on the vertices by left multiplication on  left cosets, and this action extends to the whole complex.
	\end{definition}

If all the local groups $G_v$ are cyclic of order $2$, then $\Gamma\cG$ is a right-angled Coxeter group (generally denoted $W_\Gamma$), and $D(\Gamma, \cG)$   is the standard Davis complex associated to a Coxeter group in that case.  The Davis complex associated to a general graph product has a similarly rich combinatorial geometry. More precisely:

\begin{thm}[{\cite{DavisBuildingsCAT0}}]\label{Davis_building}
The Davis complex $D(\Gamma, \cG)$  is a CAT(0) cube complex. 
\qed
\end{thm}


We  mention here a few useful observations about the action of $\Gamma\cG$ on $D(\Gamma, \cG)$.

\begin{obs}\label{obs:stab} The following holds:
	\begin{itemize}
		\item The action is without inversions, that is, an element of $\Gamma\cG$ fixing a cube of $D(\Gamma, \cG)$ globally fixes that cube pointwise.
		\item The stabiliser of a vertex corresponding to a coset $g\langle \Lambda \rangle$ is the subgroup  $g\langle \Lambda \rangle g^{-1}$ .
		 \item The action of $\Gamma\cG$ on $D(\Gamma, \cG)$ is cocompact, and a strict fundamental domain $K$ for this action is given by the subcomplex spanned by cosets of the form $\langle \Lambda\rangle $ (that is, cosets associated to the identity element of $\Gamma\cG$). 
		\end{itemize}
	\end{obs}

The CAT(0) geometry of the Davis complex can be used for instance to recover the structure of finite subgroups of graph products due to Green \cite{GreenGP}, which will be used in Section \ref{section:applicationfinite}:

\begin{lemma}[{\cite[Lemma 4.5]{GreenGP}}]\label{lem:finitesub}
	 A finite subgroup of $\Gamma \mathcal{G}$ is contained in  a conjugate of the form $g\langle \Lambda \rangle g^{-1}$ for $g \in \Gamma \mathcal{G}$ and $\Lambda \subset \Gamma$ a complete subgraph.  
\end{lemma}

\begin{proof}
	Let $H \leq \Gamma \mathcal{G}$ be a finite subgroup.  Since $D(\Gamma, \cG)$ is a (finite-dimensional, hence complete) CAT(0) cube complex by Theorem \ref{Davis_building}, the CAT(0) fixed-point theorem \cite[Corollary II.2.8]{BridsonHaefliger} implies that $H$ fixes a point of $D(\Gamma, \mathcal{G})$. Since the action is without inversions, it follows that $H$ fixes a vertex, that is, $H$ is contained in a conjugate of the form $g\langle \Lambda \rangle g^{-1}$ for $g \in \Gamma \mathcal{G}$ and $\Lambda \subset \Gamma$ a complete subgraph. 
\end{proof}
	
	\begin{remark}
		It is also possible to give a proof of Lemma \ref{lem:finitesub} using the geometry of the quasi-median graph $X(\Gamma,\mathcal{G})$ introduced in Section \ref{section:CubicalLikeGeom}. Indeed, if $H$ is a finite subgroup of  $\Gamma \mathcal{G}$, then it follows from \cite[Theorem 2.115]{Qm} that $H$ stabilises a prism of  $X(\Gamma,\mathcal{G})$, so that the conclusion of Lemma \ref{lem:finitesub} then follows from Lemma \ref{lem:PrismGP}.
	\end{remark}

\subsection{Relation between the Davis complex and the quasi-median graph}\label{sec:dual}

There is a very close link between the Davis complex and the quasi-median graph $X(\Gamma,\mathcal{G})$, which we now explain. \\

Let us consider the set of $\Gamma\cG$-translates of the fundamental domain $K$ introduced in Section \ref{section:Davis}. Two such translates  $gK$ and $hK$ share a codimension $1$ face if and only if $gh^{-1} $ is a non-trivial element of some vertex group $G_v$. Thus, the set of such translates can be seen as a chamber system over the vertex set $V(\Gamma)$ that defines a building with underlying Coxeter group $W_\Gamma$, as explained in \cite[Paragraph 5]{DavisBuildingsCAT0}. Moreover, the Davis complex is then quasi-isometric to the geometric realisation  of that  building (see \cite[Paragraph 8]{DavisBuildingsCAT0} for the geometric realisation of a building).

When thought of as a complex of chambers associated to a building, a very interesting graph associated to the Davis complex is its dual graph. This is the simplicial graph whose vertices correspond to the $\Gamma\cG$-translates of the fundamental domain $K$, and such that two vertices $gK$ and $hK$ are joined by an edge if and only if the translates share a codimension $1$ face, or in other words if $gh^{-1} $ is a non-trivial element of some local group $G_v$.  Since the action of $\Gamma\cG$ is free and transitive on the set of $\Gamma\cG$-translates of $K$, this dual graph can also be described as the simplicial graph whose vertices are elements of $\Gamma\cG$ and such that two elements $g, g'\in \Gamma\cG$ are joined by an edge if and only if there is some non-trivial element $s \in G_v$ of some vertex group such that $g' = gs$. Thus, this dual graph is exactly the quasi-median graph $X(\Gamma, \cG)$. \\

\section{Descriptions of automorphism groups}

\noindent
In this section, we study an important class of automorphisms of a graph product of groups. \textbf{Until the end of Section \ref{section:algebraiccharacterisation}, we fix a finite simplicial graph $\Gamma$ and a collection of groups $\mathcal{G}$  indexed by $V(\Gamma)$.} Applications to particular classes of graphs products will then be given in Sections \ref{section:applicationfinite} and \ref{section:applicationgirth}.

\begin{definition}
A \emph{conjugating automorphism} of $\Gamma \mathcal{G}$ is an automorphism $\varphi : \Gamma \mathcal{G} \to \Gamma \mathcal{G}$ satisfying the following property: For every $G \in \mathcal{G}$, there exist $H \in \mathcal{G}$ and $g \in \Gamma \mathcal{G}$ such that $\varphi(G)=gHg^{-1}$. 
	Let $\mathrm{ConjAut}_{\Gamma, \cG}(\Gamma \mathcal{G})$ be the subgroup of conjugating automorphisms of $\Gamma\cG$. In order to lighten notations, we will simply denote it $\mathrm{ConjAut}(\Gamma \mathcal{G})$ in the rest of this article, by a slight abuse of notation  that we comment on below.
\end{definition}

\begin{remark} 
We emphasize that the subgroup $\mathrm{ConjAut}_{\Gamma, \cG}(\Gamma \mathcal{G})$ of $\mathrm{Aut}(\Gamma \mathcal{G})$ heavily depends on the chosen decomposition of $\Gamma \mathcal{G}$ as a graph product under consideration, and not just on the group $\Gamma \mathcal{G}$ itself. For instance, let $\Gamma \mathcal{G}$ be the graph product corresponding to a single edge whose endpoints are both labelled by $\mathbb{Z}$, and let $\Phi \mathcal{H}$ be the graph product corresponding to a single vertex labelled by $\mathbb{Z}^2$. Both $\Gamma \mathcal{G}$ and $\Phi \mathcal{H}$ are isomorphic to $\mathbb{Z}^2$, but $\mathrm{ConjAut}_{\Gamma, \cG}(\Gamma \mathcal{G})$ is finite while $\mathrm{ConjAut}_{\Phi, \mathcal{H}}(\Phi \mathcal{H}) = \mathrm{GL}(2, \mathbb{Z})$.

Thus, writing $\mathrm{ConjAut}(\Gamma \mathcal{G})$  instead of $\mathrm{ConjAut}_{\Gamma, \cG}(\Gamma \mathcal{G})$  is indeed an abuse of notation. However, in this article we will only consider a single graph product decomposition at a time, so this lighter notation  shall not lead to confusion.
\end{remark}

Our goal in this section is to find a simple and natural generating set for $\mathrm{ConjAut}(\Gamma \mathcal{G})$. For this purpose we need the following definitions:

\begin{definition}
We define the following automorphisms of $\Gamma\cG$:
\begin{itemize}
	\item Given an isometry $\sigma : \Gamma \to \Gamma$ and a collection of isomorphisms $\Phi = \{ \varphi_u : G_u \to G_{\sigma(u)} \mid u \in V(\Gamma) \}$, the \emph{local automorphism} $(\sigma, \Phi)$ is the automorphism of $\Gamma \mathcal{G}$ induced by $$\left\{ \begin{array}{ccc} \bigcup\limits_{u \in V(\Gamma)} G_u & \to & \Gamma \mathcal{G} \\ g & \mapsto & \text{$\varphi_u(g)$ if $g \in G_u$} \end{array} \right..$$ The group of local automorphisms of $\Gamma \mathcal{G}$ is denoted by $\mathrm{Loc}(\Gamma \mathcal{G})$. Again, this subgroup should be denoted $\mathrm{Loc}_{\Gamma, \cG}(\Gamma \mathcal{G})$ as it depends on the chosen decomposition as a graph product, but we will use the same abuse of notation as above. Also, we denote by $\mathrm{Loc}^0(\Gamma \mathcal{G})$ the group of local automorphisms satisfying $\sigma = \mathrm{Id}$. Notice that $\mathrm{Loc}^0(\Gamma \mathcal{G})$ is a finite-index subgroup of $\mathrm{Loc}(\Gamma \mathcal{G})$ naturally isomorphic to the direct sum $\bigoplus\limits_{u \in V(\Gamma)} \mathrm{Aut}(G_u)$. 
	\item Given a vertex $u \in V(\Gamma)$, a connected component $\Lambda$ of $\Gamma \backslash \mathrm{star}(u)$ and an element $h \in G_u$, the \emph{partial conjugation} $(u, \Lambda,h)$ is the automorphism of $\Gamma \mathcal{G}$ induced by $$\left\{ \begin{array}{ccc} \bigcup\limits_{u \in V(\Gamma)} G_u & \to & \Gamma \mathcal{G} \\ g & \mapsto & \left\{ \begin{array}{cl} g & \text{if $g \notin \langle \Lambda \rangle$} \\ hgh^{-1} & \text{if $g \in \langle \Lambda \rangle$} \end{array} \right. \end{array} \right. .$$ Notice that an inner automorphism of $\Gamma \mathcal{G}$ is always a product of partial conjugations. 
\end{itemize}
We denote by $\mathrm{ConjP}(\Gamma \mathcal{G})$ the subgroup of $\mathrm{Aut}(\Gamma \mathcal{G})$ generated by the inner automorphisms, the local automorphisms, and the partial conjugations. Again, this subgroup should be denoted $\mathrm{ConjP}_{\Gamma, \cG}(\Gamma \mathcal{G})$ as it depends on the chosen decomposition as a graph product, but we will use the same abuse of notation as above.
\end{definition}

\noindent
It is clear that the inclusion $\mathrm{ConjP}(\Gamma \mathcal{G}) \subset \mathrm{ConjAut}(\Gamma \mathcal{G})$ holds. The main result of this section is that the reverse inclusion also holds. Namely:

\begin{thm}\label{thm:conjugatingauto}
The group of conjugating automorphisms of $\Gamma \mathcal{G}$ coincides with $\mathrm{ConjP}(\Gamma \mathcal{G})$. 
\end{thm}

\noindent
In Sections \ref{section:applicationfinite} and \ref{section:applicationgirth}, we deduce a description of automorphism groups of specific classes of graph products.

\subsection{Action of $\mathrm{ConjAut}(\Gamma \mathcal{G})$ on the transversality graph}\label{section:algebraiccharacterisation}

\noindent
In this section,  our goal is to extract from the quasi-median graph associated to a given graph product a graph on which the group of conjugating automorphisms acts.  \textbf{Let us recall that, until the end of Section \ref{section:algebraiccharacterisation}, we fix a finite simplicial graph $\Gamma$ and a collection of groups $\mathcal{G}$  indexed by $V(\Gamma)$.} 

\begin{definition}
The \emph{transversality graph} $T(\Gamma, \mathcal{G})$ is the graph whose vertices are the hyperplanes of $X(\Gamma, \mathcal{G})$ and whose edges connect two hyperplanes whenever they are transverse. 
\end{definition}

Note that the transversality graph is naturally isomorphic to the crossing graph of the Davis complex $D(\Gamma, \cG)$, that is, the simplicial graph whose vertices are parallelism classes of hyperplanes of $D(\Gamma, \cG)$ and whose edges correspond to transverse (classes of) hyperplanes. One of the advantages of working with the transversality graph instead is that it does not require to talk about parallelism classes or to choose particular representatives in proofs, which will make some of the arguments in this section simpler.

\noindent
From this definition, it is not clear at all that the group of conjugating automorphisms acts on the corresponding transversality graph. To solve this problem, we will state and prove an algebraic characterisation of the transversality graph. This description is the following:

\begin{definition}
The \emph{factor graph} $F(\Gamma, \mathcal{G})$ is the graph whose vertices are the conjugates of the vertex groups and whose edges connect two conjugates whenever they commute (ie., every element of one subgroup commutes with every element of the other one).
\end{definition}

\noindent
The main result of this section is the following algebraic characterisation:

\begin{prop}\label{prop:factorgraphalg}
The map
$$\left\{ \begin{array}{ccc} T(\Gamma, \mathcal{G}) & \to & F(\Gamma, \mathcal{G}) \\ J & \mapsto & \mathrm{stab}_\circlearrowleft(J) \end{array} \right.$$
induces a graph isomorphism $T(\Gamma, \mathcal{G}) \to F(\Gamma, \mathcal{G})$.
\end{prop}

\begin{proof}
Because the rotative-stabilisers of a hyperplane is indeed a conjugate of a vertex group, according to Lemma \ref{lem:CliqueStab} and Proposition\ref{prop:rotstabinX}, our map is well-defined. Let $G \in \mathcal{G}$ be a vertex group and let $g \in \Gamma \mathcal{G}$. Then $gGg^{-1}$ is the stabiliser of the clique $gG$, and we deduce from Proposition \ref{prop:rotstabinX} that it is also the rotative-stabilisers of the hyperplane dual to $gG$. Consequently, our map is surjective. To prove its injectivity, it is sufficient to show that two distinct hyperplanes $J_1,J_2$ which are not transverse have different rotative-stabilisers. More generally, we want to prove the following observation:

\begin{fact}\label{fact:disctingrotativestab}
The rotative-stabilisers of two distinct hyperplanes $J_1,J_2$ of $X(\Gamma, \mathcal{G})$ have a trivial intersection.
\end{fact}

\noindent
Indeed, if we fix a clique $C$ dual to $J_2$, then it must be entirely contained in a sector delimited by $J_1$. But, if $g \in \mathrm{stab}_\circlearrowleft(J_1)$ is non-trivial, then it follows from Proposition \ref{prop:rotstabinX} that $g$ does not stabilise the sector delimited by $J_1$ which contains $C$. Hence $gJ_2 \neq J_2$. A fortiori, $g$ does not belong to the rotative-stabilisers of $J_2$, proving our fact.

\medskip \noindent
In order to conclude the proof of our proposition, it remains to show that two distinct hyperplanes $J_1$ and $J_2$ of $X(\Gamma,\mathcal{G})$ are transverse if and only if their rotative-stabilisers commute. 

\medskip \noindent
Suppose first that $J_1$ and $J_2$ are not transverse. Let $g \in \mathrm{stab}_\circlearrowleft(J_1)$ and $h \in \mathrm{stab}_\circlearrowleft(J_2)$ be two non-trivial elements. Since we know from Proposition \ref{prop:rotstabinX} that $g$ does not stabilise the sector delimited by $J_1$ which contains $J_2$, necessarily $J_1$ separates $J_2$ and $gJ_2$. Similarly, we deduce that $J_2$ separates $J_1$ and $hgJ_2$; and that $J_1$ separates $J_2$ and $gJ_2$. Therefore, both $J_1$ and $J_2$ separate $hgJ_2$ and $gJ_2=ghJ_2$. A fortiori, $hgJ_2 \neq ghJ_2$ so that $gh \neq hg$. Thus, we have proved that the rotative-stabilisers of $J_1$ and $J_2$ do not commute.

\medskip \noindent
Next, suppose that $J_1$ and $J_2$ are transverse. Up to translating by an element of $\Gamma \mathcal{G}$, we may suppose suppose without loss of generality that the vertex $1$ belongs to $N(J_1) \cap N(J_2)$. As a consequence, there exist vertices $u,v \in V(\Gamma)$ such that $J_1=J_u$ and $J_2=J_v$. We know from Lemma \ref{lem:transverseimpliesadj} that $u$ and $v$ are adjacent. Therefore, $\mathrm{stab}_\circlearrowleft(J_1)= G_u$ and $\mathrm{stab}_{\circlearrowleft} (J_2)= G_v$ commute.
\end{proof}

\noindent
Interestingly, if $\varphi : \Gamma \mathcal{G} \to \Phi \mathcal{H}$ is an isomorphism between two graph products that sends vertex groups to conjugates of vertex groups, then $\varphi$ naturally defines an isometry
$$\left\{ \begin{array}{ccc} F(\Gamma, \mathcal{G}) & \to & F(\Phi,\mathcal{H}) \\ H & \mapsto & \varphi(H) \end{array} \right..$$
By transferring this observation to transversality graphs through the isomorphism providing by Proposition \ref{prop:factorgraphalg}, one gets the following statement:

\begin{fact}\label{fact:ConjIsomacts}
Let $\Gamma, \Phi$ be two simplicial graphs and $\mathcal{G}, \mathcal{H}$ two collections of groups indexed by $V(\Gamma),V(\Phi)$ respectively. Suppose that $\varphi  : \Gamma \mathcal{G} \to \Phi \mathcal{H}$ is an isomorphism between two graph products which sends vertex groups to conjugates of vertex groups. Then $\varphi$ defines an isometry $T(\Gamma, \mathcal{G}) \to T(\Phi, \mathcal{H})$ via
$$J \mapsto \text{hyperplane whose rotative-stabiliser is $\varphi(\mathrm{stab}_\circlearrowleft(J))$}.$$
\end{fact}

\noindent
In the case $\Gamma= \Phi$, by transferring the action $\mathrm{ConjAut}(\Gamma \mathcal{G}) \curvearrowright F(\Gamma, \mathcal{G})$ defined by:
$$\left\{ \begin{array}{ccc} \mathrm{ConjAut}(\Gamma \mathcal{G}) & \to & \mathrm{Isom}(F(\Gamma, \mathcal{G})) \\ \varphi & \mapsto & \left( H \mapsto \varphi(H) \right) \end{array} \right.$$
to the transversality graph $T(\Gamma, \mathcal{G})$, one gets the following statement:

\begin{fact}\label{fact:ConjAutacts}
The group $\mathrm{ConjAut}(\Gamma \mathcal{G})$ acts on the transversality graph $T(\Gamma, \mathcal{G})$ via
$$\left\{ \begin{array}{ccc} \mathrm{ConjAut}(\Gamma \mathcal{G}) & \to & \mathrm{Isom}(T(\Gamma, \mathcal{G})) \\ \varphi & \mapsto & \left( J \mapsto \text{hyperplane whose rotative-stabiliser is $\varphi( \mathrm{stab}_\circlearrowleft(J))$} \right) \end{array} \right. .$$
Moreover, if $\Gamma \mathcal{G}$ is centerless and if we identify the subgroup of inner automorphism $\mathrm{Inn}(\Gamma \mathcal{G}) \leq \mathrm{ConjAut}(\Gamma \mathcal{G})$ with $\Gamma \mathcal{G}$ via
$$\left\{ \begin{array}{ccc} \Gamma \mathcal{G} & \to & \mathrm{Inn}(\Gamma \mathcal{G}) \\ g & \mapsto & (x \mapsto gxg^{-1}) \end{array} \right.,$$
then the action $\mathrm{ConjAut}(\Gamma \mathcal{G}) \curvearrowright T(\Gamma, \mathcal{G})$ extends the natural action $\Gamma \mathcal{G} \curvearrowright T(\Gamma, \mathcal{G})$ induced by $\Gamma \mathcal{G} \curvearrowright X(\Gamma,\mathcal{G})$.
\end{fact}

\subsection{Characterisation of conjugating automorphisms}\label{section:ConjAut}

\noindent
This section is dedicated to the proof of Theorem \ref{thm:conjugatingauto}, which will be based on the transversality graph introduced in the previous section. In fact, we will prove a stronger statement, namely:

\begin{thm}\label{thm:ConjIsom}
Let $\Gamma, \Phi$ be two simplicial graphs and $\mathcal{G},\mathcal{H}$ two collections of groups indexed by $V(\Gamma),V(\Phi)$ respectively. Suppose that there exists an isomorphism $\varphi : \Gamma \mathcal{G} \to \Phi \mathcal{H}$ which sends vertex groups of $\Gamma \mathcal{G}$ to conjugates of vertex groups of $\Phi \mathcal{H}$. Then there exist an automorphism $\alpha \in \mathrm{ConjP}(\Gamma \mathcal{G})$ and an isometry $s : \Gamma \to \Phi$ such that $\varphi \alpha$ sends isomorphically $G_u$ to $H_{s(u)}$ for every $u \in V(\Gamma)$. 
\end{thm}

It should be noted that the proof of Theorem \ref{thm:ConjIsom} we give below is constructive. Namely, given two graph products $\Gamma \mathcal{G}$ and $\Phi \mathcal{H}$, it is possible to construct  an algorithm that determines whether a given isomorphism $\varphi : \Gamma\mathcal{G} \to \Phi \mathcal{H}$ 
sends vertex groups to conjugates of vertex groups. In such a case, that algorithm also  produces a list of partial conjugations $(u_1, \Lambda_1,h_1)$, $\ldots$, $(u_k, \Lambda_k,h_k) \in \mathrm{Aut}(\Gamma \mathcal{G})$ and an isometry $s: \Gamma \to \Phi$ such that $\varphi \circ (u_1, \Lambda_1,h_1) \cdots (u_k, \Lambda_k,h_k)$ sends $G_u$ isomorphically to $H_{s(u)}$ for every $u\in V(\Gamma)$. However, since such algorithmic considerations are beyond the scope of this article, we will not develop this further.\\

Recall from Fact \ref{fact:ConjIsomacts} that such an isomorphism $\varphi : \Gamma \mathcal{G} \to \Phi \mathcal{H}$ defines an isometry $T(\Gamma, \mathcal{G}) \to T(\Phi, \mathcal{H})$ via
$$J  \mapsto \text{hyperplane whose rotative-stabiliser is $\varphi(\mathrm{stab}_\circlearrowleft(J))$}.$$
By setting $\mathcal{J}= \{ J_u \mid u \in V(\Gamma)\}$, we define the \emph{complexity} of $\varphi$ by
$$\| \varphi \| = \sum\limits_{J \in \mathcal{J}} d(1, N(\varphi \cdot J)),$$
where $N(\cdot)$ denotes the carrier of a hyperplane; see Definition \ref{def:hyp}. Theorem \ref{thm:ConjIsom} will be an easy consequence of the following two lemmas:

\begin{lemma}\label{lem:shortencomplexity}
Let $\varphi$ be as in the statement of Theorem \ref{thm:ConjIsom}. If $\| \varphi \| \geq 1$ then there exists some automorphism $\alpha \in \mathrm{ConjP}(\Gamma \mathcal{G})$ such that $\| \varphi \alpha \| < \| \varphi \|$.
\end{lemma} 

\begin{lemma}\label{lem:complexityzero}
Let $\varphi$ be as in the statement of Theorem \ref{thm:ConjIsom}. If $\| \varphi \| =0$ then there exists an isometry $s : \Gamma \to \Phi$ such that $\varphi$ sends isomorphically $G_u$ to $H_{s(u)}$ for every $u \in V(\Gamma)$. 
\end{lemma}

\noindent
Before turning to the proof of these two lemmas, we need the following observation:

\begin{claim}\label{claim:peripheral}
Let $\varphi$ be as in the statement of Theorem \ref{thm:ConjIsom}. If $\varphi \cdot \mathcal{J}$ is peripheral (see Definition \ref{def:peripheral}) then $\| \varphi \| =0$. 
\end{claim}

\begin{proof}
Notice that 
$$\langle \mathrm{stab}_\circlearrowleft (J) \mid J \in \varphi \cdot \mathcal{J} \rangle= \varphi \left( \langle \mathrm{stab}_\circlearrowleft (J) \mid J \in \mathcal{J} \rangle \right) = \Phi \mathcal{H}.$$
We deduce from Lemma \ref{lem:pingpong} that, if we denote by $S(J)$ the sector delimited by $\varphi \cdot J$ which contains $1$, then $g \notin \bigcap\limits_{J \in \varphi \cdot \mathcal{J}} S(J)$ for every $g \in \Phi \mathcal{H} \backslash \{ 1 \}$. Since the action $\Phi \mathcal{H} \curvearrowright X(\Phi, \mathcal{H})$ is vertex-transitive, it follows that $\bigcap\limits_{J \in \varphi \cdot \mathcal{J}} S(J)= \{1\}$, which implies that $1 \in \bigcap\limits_{J \in \varphi \cdot \mathcal{J}} N(J)$, and finally that $\| \varphi \|=0$, proving our claim.
\end{proof}

\begin{proof}[Proof of Lemma \ref{lem:shortencomplexity}.]
We deduce from Claim \ref{claim:peripheral} that $\varphi \cdot \mathcal{J}$ is not peripheral, ie., there exist two distinct vertices $a,b \in V(\Gamma)$ such that $\varphi \cdot J_a$ separates $1$ from $\varphi \cdot J_b$. Notice that $a$ and $b$ are not adjacent in $\Gamma$ as the hyperplanes $\varphi \cdot J_a$ and $\varphi \cdot J_b$ are not transverse. Let $x \in \mathrm{stab}_\circlearrowleft (\varphi \cdot J_a)$ be the element sending $\varphi \cdot J_b$ into the sector delimited by $\varphi \cdot J_a$ which contains $1$. Notice that $\mathrm{stab}_\circlearrowleft (\varphi \cdot J_a)= \varphi \left( \mathrm{stab}_\circlearrowleft(J_a) \right)$, so there exists some $y \in \mathrm{stab}_\circlearrowleft(J_a)$ such that $x= \varphi(y)$. 

\medskip \noindent
Now, let $\alpha$ denote the partial conjugation of $\Gamma \mathcal{G}$ which conjugates by $y$ the vertex groups of the connected component $\Lambda$ of $\Gamma \backslash \mathrm{star}(a)$ which contains $b$. According to Fact \ref{fact:ConjAutacts}, $\alpha$ can be thought of as an isometry of $T(\Gamma, \mathcal{G})$. We claim that $\| \varphi \alpha \| < \| \varphi \|$. Notice that $\varphi \alpha \cdot J_u= \varphi \cdot J_u$ for every $u \notin \Lambda$, and that
$$\varphi \alpha \cdot J_u = \varphi \cdot yJ_u = \varphi(y) \cdot \varphi \cdot J_u = x \cdot \varphi \cdot J_u$$
for every $u \in \Lambda$. Therefore, in order to prove the inequality $\| \varphi \alpha \| < \| \varphi \|$, it is sufficient to show that $d(1,N(x \cdot \varphi \cdot J_u))< d(1,N(\varphi \cdot J_u))$ for every $u \in \Lambda$. It is a consequence of Lemma \ref{lem:shortendist} and of the following observation:

\begin{fact}
For every $u \in \Lambda$, $\varphi \cdot J_u$ is contained in the sector delimited by $\varphi \cdot J_a$ which contains $\varphi \cdot J_b$.
\end{fact}

\noindent
Let $u \in \Lambda$ be a vertex. By definition of $\Lambda$, there exists a path 
$$u_1=u, \ u_2, \ldots, u_{n-1}, \ u_n=b$$
in $\Gamma$ that is disjoint from $\mathrm{star}(a)$. This yields a path
$$\varphi \cdot J_{u_1}=\varphi \cdot J_u, \ \varphi \cdot J_{u_2}, \ldots, \varphi \cdot J_{u_{n-1}}, \ \varphi \cdot J_{u_n}=\varphi \cdot J_b$$
in the transversality graph $T(\Gamma, \mathcal{G})$. As a consequence of Lemma \ref{lem:transverseimpliesadj}, none of these hyperplanes are transverse to $\varphi \cdot J_a$, which implies that they are all contained in the same sector delimited by $\varphi \cdot J_a$, namely the one containing $\varphi \cdot J_b$, proving our fact. This concludes the proof of our lemma.
\end{proof}

\begin{proof}[Proof of Lemma \ref{lem:complexityzero}.]
If $\| \varphi \|=0$ then $1 \in \bigcap\limits_{J \in \varphi \cdot \mathcal{J}} N(J)$, which implies that $\varphi \cdot \mathcal{J} \subset \mathcal{K}$ where $\mathcal{K}= \{ J_u \mid u \in V(\Phi) \}$. Let $\Lambda$ be the subgraph of $\Phi$ such that $\varphi \cdot \mathcal{J}= \{ J_u \mid u \in V(\Lambda)$. Notice that 
$$\langle \mathrm{stab}_\circlearrowleft (J) \mid J \in \varphi \cdot \mathcal{J} \rangle= \varphi \left( \langle \mathrm{stab}_\circlearrowleft (J) \mid J \in \mathcal{J} \rangle \right) = \Phi \mathcal{H}.$$
Therefore, one has
$$G_u= \mathrm{stab}_\circlearrowleft (J_u) \leq \Phi \mathcal{H} = \langle \mathrm{stab}_\circlearrowleft(J) \mid J \in \varphi \cdot \mathcal{J} \rangle = \langle \Lambda \rangle$$
for every $u \in V(\Phi)$, hence $\Lambda = \Phi$. This precisely means that $\varphi \cdot \mathcal{J} = \mathcal{K}$. In other words, there exists a map $s : V(\Gamma) \to V(\Phi)$ such that $\varphi$ sends isomorphically $G_u$ to $H_{s(u)}$ for every $u \in V(\Gamma)$. Notice that $s$ necessarily defines an isometry since $\varphi$ sends two (non-)transverse hyperplanes to two (non-)transverse hyperplanes.
\end{proof}

\begin{proof}[Proof of Theorem \ref{thm:ConjIsom}.]
Given our isomorphism $\varphi : \Gamma \mathcal{G} \to \Phi \mathcal{H}$, we apply Lemma \ref{lem:shortencomplexity} iteratively to find automorphisms $\alpha_1, \ldots, \alpha_m \in \mathrm{ConjP}(\Gamma \mathcal{G})$ such that $\| \varphi \alpha_1 \cdots \alpha_m \| =0$. Set $\alpha = \alpha_1 \cdots \alpha_m$. By Lemma \ref{lem:complexityzero},  there exists an isometry $s : \Gamma \to \Phi$ such that $\varphi$ sends isomorphically $G_u$ to $H_{s(u)}$ for every $u \in V(\Gamma)$. 
\end{proof}

\begin{proof}[Proof of Theorem \ref{thm:conjugatingauto}.]
An immediate consequence of Theorem \ref{thm:ConjIsom} is that $\mathrm{ConjAut}(\Gamma \mathcal{G})= \mathrm{Loc}(\Gamma \mathcal{G}) \cdot \mathrm{ConjP}(\Gamma \mathcal{G})$. As $\mathrm{Loc}(\Gamma \mathcal{G})$ is contained in $\mathrm{ConjP}(\Gamma \mathcal{G})$, it follows that $\mathrm{ConjP}(\Gamma \mathcal{G})$ and $\mathrm{ConjAut}(\Gamma \mathcal{G})$ coincide.
\end{proof}

\subsection{Application to graph products of finite groups}\label{section:applicationfinite}

\noindent
In this section, we focus on graph products of finite groups. This includes, for instance, right-angled Coxeter groups. The main result of this section, which will be deduced from Theorem \ref{thm:conjugatingauto}, is the following:

\begin{thm}\label{thm:GPfinitegroups}
Let $\Gamma$ be a finite simplicial graph and $\mathcal{G}$ a collection of finite groups indexed by $V(\Gamma)$. Then the subgroup generated by partial conjugations has finite index in $\mathrm{Aut}(\Gamma \mathcal{G})$. 
\end{thm}

\noindent
Before turning to the proof of the theorem, we will need the following preliminary result about general graph products of groups: 

\begin{lemma}\label{lem:conjincluded}
Let $\Gamma$ be a simplicial graph and $\mathcal{G}$ a collection of groups indexed by $V(\Gamma)$. Fix two subgraphs $\Lambda_1,\Lambda_2 \subset \Gamma$ and an element $g \in \Gamma \mathcal{G}$. If $\langle \Lambda_1 \rangle \subset g \langle \Lambda_2 \rangle g^{-1}$, then $\Lambda_1 \subset \Lambda_2$ and $g \in \langle \Lambda_1 \rangle \cdot \langle \mathrm{link}(\Lambda_1) \rangle \cdot \langle \Lambda_2 \rangle$.
\end{lemma}

\begin{proof}
Notice that the subgroups $\langle \Lambda_1 \rangle$ and $g \langle \Lambda_2 \rangle g^{-1}$ coincide with the stabiliser of the subgraphs $\langle \Lambda_1 \rangle$ and $g \langle \Lambda_2 \rangle$ respectively. We claim that any hyperplane intersecting $\langle \Lambda_1 \rangle$ also intersects $g \langle \Lambda_2 \rangle$.

\medskip \noindent
Suppose by contradiction that this is not the case, ie., there exists some hyperplane $J$ intersecting $\langle \Lambda_1 \rangle$ but not $g \langle \Lambda_2 \rangle$. Fix a clique $C \subset \langle \Lambda_1 \rangle$ dual to $J$. There exist $u \in \Lambda_1$ and $h \in \langle \Lambda_1 \rangle$ such that $\mathrm{stab}(C)=h \langle u \rangle h^{-1}$. As a consequence, $\mathrm{stab}(C)$, which turns out to coincide with the rotative stabiliser of $J$ according to Proposition \ref{prop:rotstabinX}, stabilises $\langle \Lambda_1 \rangle$. On the other hand, a non-trivial element of $\mathrm{stab}_{\circlearrowleft}(J)$ sends $g \langle \Lambda_2 \rangle$ into a sector that  does not contain $g \langle \Lambda_2 \rangle$: a fortiori, it does not stabilise $g \langle \Lambda_2 \rangle$, contradicting our assumptions. This concludes the proof of our claim.

\medskip \noindent
Because the cliques dual to the hyperplanes intersecting $\langle \Lambda_1 \rangle$ are labelled by vertices of $\Lambda_1$, and similarly that the cliques dual to the hyperplanes intersecting $g \langle \Lambda_2 \rangle$ are labelled by vertices of $\Lambda_2$, we deduce from the previous claim that $\Lambda_1 \subset \Lambda_2$, proving the first assertion of our lemma. 

\medskip \noindent
Next, let $x \in \langle \Lambda_1 \rangle$ and $y \in g \langle \Lambda_2 \rangle$ be two vertices minimising the distance between $\langle \Lambda_1 \rangle$ and $g \langle \Lambda_2 \rangle$. Fix a path $\alpha$ from $1$ to $x$ in $\langle \Lambda_1 \rangle$, a geodesic $\beta$ between $x$ and $y$, and a path $\gamma$ from $y$ to $g$ in $g \langle \Lambda_2 \rangle$. Thus, $g=abc$ where $a,b,c$ are the words labelling the paths $\alpha,\beta,\gamma$. Notice that $a \in \langle \Lambda_1 \rangle$ and $c \in \langle \Lambda_2 \rangle$. The labels of the edges of $\beta$ coincide with the labels of the hyperplanes separating $x$ and $y$, or equivalently (according to Corollary \ref{cor:minseppairhyp}), to the labels of the hyperplanes separating $\langle \Lambda_1 \rangle$ and $g \langle \Lambda_2 \rangle$. But we saw that any hyperplane intersecting $\langle \Lambda_1 \rangle$ intersects $g \langle \Lambda_2 \rangle$ as well, so any hyperplane separating $\langle \Lambda_1 \rangle$ and $g \langle \Lambda_2 \rangle$ must be transverse to any hyperplane intersecting $\langle \Lambda_1 \rangle$. Because the set of  labels of the hyperplanes of $\langle \Lambda_1 \rangle$ is $V(\Lambda_1)$, we deduce from Lemma \ref{lem:transverseimpliesadj} that the vertex of $\Gamma$ labelling an edge of $\beta$ is adjacent to all the vertices of $\Lambda_1$, ie., it belongs to $\mathrm{link}(\Lambda_1)$. Thus, we have proved that $b \in \langle \mathrm{link}(\Lambda_1) \rangle$, hence $g=abc \in \langle \Lambda_1 \rangle \cdot \langle \mathrm{link}(\Lambda_1) \rangle \cdot \langle \Lambda_2 \rangle$. 
\end{proof}

%
%

We now move to the proof of Theorem \ref{thm:GPfinitegroups}. \textbf{For the rest of Section \ref{section:applicationfinite}, we fix a finite simplicial graph $\Gamma$ and a collection of finite groups $\mathcal{G}$  indexed by $V(\Gamma)$.}

\begin{lemma}\label{lem:maxfinitesub}
The maximal finite subgroups of $\Gamma \mathcal{G}$ are exactly the $g \langle \Lambda \rangle g^{-1}$ where $g \in \Gamma \mathcal{G}$ and where $\Lambda \subset \Gamma$ is a maximal complete subgraph of $\Gamma$. 
\end{lemma}

\begin{proof}
Let $H \leq \Gamma \mathcal{G}$ be a maximal finite subgroup. By Lemma \ref{lem:finitesub}, there exist $g \in \Gamma \mathcal{G}$ and a complete subgraph  $\Lambda \subset \Gamma$ such that $H$ is contained in $g \langle \Lambda \rangle g^{-1}$. 
Since $H$ is a maximal finite subgroup, then $H= g \langle \Lambda \rangle g^{-1}$ since $g \langle \Lambda \rangle g^{-1}$ is clearly finite. Moreover, the maximality of $H$ implies that $\Lambda$ is a maximal complete subgraph of $\Gamma$. 

\medskip \noindent
Conversely, let $g \in \Gamma \mathcal{G}$ be an element and $\Lambda \subset \Gamma$ a complete subgraph. Clearly, $g\langle \Lambda \rangle g^{-1}$ is finite. Now, suppose that $\Lambda$ is a maximal complete subgraph of $\Gamma$. If $F$ is a finite subgroup of $\Gamma \mathcal{G}$ containing $g \langle \Lambda \rangle g^{-1}$, we know that $F= h \langle \Xi \rangle h^{-1}$ for some $h \in \Gamma \mathcal{G}$ and for some complete subgraph $\Xi \subset \Gamma$, so that we deduce from the inclusion
$$g \langle \Lambda \rangle g^{-1} \leq F = h \langle \Xi \rangle h^{-1}$$
and from Lemma \ref{lem:conjincluded} that $\Lambda \subset \Xi$. By maximality of $\Lambda$, necessarily $\Lambda = \Xi$. We also deduce from Lemma \ref{lem:conjincluded} that $h \in g \langle \Lambda \rangle \cdot \langle \mathrm{link}(\Lambda) \rangle \cdot \langle \Lambda \rangle$. But the maximality of $\Lambda$ implies that $\mathrm{link}(\Lambda)= \emptyset$, hence $h \in g \langle \Lambda \rangle$. Therefore
$$F= h \langle \Xi \rangle h^{-1} =h \langle \Lambda \rangle h^{-1}= g \langle \Lambda \rangle g^{-1}.$$
Thus, we have proved that $g \langle \Lambda \rangle g^{-1}$ is a maximal finite subgroup of $\Gamma \mathcal{G}$. 
\end{proof}

\begin{cor}\label{cor:maxself}
The maximal finite subgroups of $\Gamma \mathcal{G}$ are self-normalising.
\end{cor}

\begin{proof}
By Lemma \ref{lem:maxfinitesub},  maximal finite subgroups are of the form $g\langle \Lambda \rangle g^{-1}$ for a maximal clique $\Lambda \subset \Gamma$. Since a maximal clique has an empty link, it follows from Lemma \ref{lem:conjincluded} that such subgroups are self-normalising.
\end{proof}

\begin{proof}[Proof of Theorem \ref{thm:GPfinitegroups}.]
As a consequence of Lemma \ref{lem:maxfinitesub}, $\Gamma \mathcal{G}$ contains only finitely many conjugacy classes of maximal finite subgroups, so $\mathrm{Aut}(\Gamma \mathcal{G})$ contains a finite-index subgroup $H$ such that, for every maximal finite subgroup $F \leq \Gamma \mathcal{G}$ and every $\varphi \in H$, the subgroups $\varphi(F)$ and $F$ are conjugate. 

\medskip \noindent
Fix a maximal finite subgroup $F \leq \Gamma \mathcal{G}$. We define a morphism $\Psi_F : H \to \mathrm{Out}(F)$ in the following way. If $\varphi \in H$, there exists some $g \in \Gamma \mathcal{G}$ such that $\varphi(F)=gFg^{-1}$. Set $\Psi_F(\varphi) = \left[ \left( \iota(g)^{-1} \varphi \right)_{|F} \right]$ where $\iota(g)$ denotes the inner automorphism associated to $g$. Notice that, if $h \in \Gamma \mathcal{G}$ is another element such that $\varphi(F)=hFh^{-1}$, then $h=gs$ for some $s \in \Gamma \mathcal{G}$ normalising $F$. In fact, since $F$ is self-normalising according to Corollary \ref{cor:maxself}, $s$ must belong to $F$. Consequently, the automorphisms $\left( \iota(g)^{-1} \varphi \right)_{|F}$ and $\left( \iota(h)^{-1} \varphi \right)_{|F}$ of $F$ have the same image in $\mathrm{Out}(F)$. The conclusion is that $\Psi_F$ is well-defined as a map $H \to \mathrm{Out}(F)$. It remains to show that it is a morphism. So let $\varphi_1,\varphi_2 \in H$ be two automorphisms, and fix two elements $g_1,g_2 \in \Gamma \mathcal{G}$ such that $\varphi_i(F)=g_iFg_i^{-1}$ for $i=1,2$. Notice that $\varphi_1\varphi_2(F)= \varphi_1(g_2) g_1^{-1} F g_1 \varphi_1(g_2)^{-1}$, so that we have:  
$$\Psi_F(\varphi_1\varphi_2) = \left[ \left( \iota(\varphi_1(g_2)g_1^{-1} )^{-1} \varphi_1 \varphi_2 \right)_{|F} \right];$$
consequently,
$$\begin{array}{lcl} \Psi_F(\varphi_1) \Psi_F(\varphi_2) & = & \left[ \left( \iota(g_1)^{-1} \varphi_1 \right)_{|F} \right] \cdot \left[ \left( \iota(g_2)^{-1} \varphi_2 \right)_{|F} \right] = \left[ \left( \iota(g_1)^{-1} \varphi_1 \iota(g_2)^{-1} \varphi_2 \right)_{|F} \right] \\ \\ & = & \left[ \left( \iota(g_1)^{-1} \iota(\varphi_1(g_2))^{-1} \varphi_1 \varphi_2 \right)_{|F} \right] = \Psi_F(\varphi_1 \varphi_2) \end{array}$$
Moreover, it is clear that $\Psi_F(\mathrm{Id})= \mathrm{Id}$. We conclude that $\Psi_F$ defines a morphism $H \to \mathrm{Out}(F)$. Now, we set
$$K= \bigcap \{ \mathrm{ker}(\Psi_F) \mid \text{$F \leq \Gamma \mathcal{G}$ maximal finite subgroup} \}.$$
Notice that $K$ is a finite-index subgroup of $\mathrm{Aut}(\Gamma \mathcal{G})$ since it is the intersection of finitely many finite-index subgroups. We want to show that $K$ is a subgroup of $\mathrm{ConjP}(\Gamma \mathcal{G})$. According to Theorem \ref{thm:conjugatingauto}, it is sufficient to show that, for every $\varphi \in K$ and every $u \in V(\Gamma)$, the subgroups $\varphi(G_u)$ and $G_u$ are conjugate. 

\medskip \noindent
So fix a vertex $u \in V(\Gamma)$ and an automorphism $\varphi \in K$. As a consequence of Lemma~\ref{lem:maxfinitesub}, there exists a maximal finite subgroup $F \leq \Gamma \mathcal{G}$ containing $G_u$. Since $\varphi$ belongs to $H$, there exists some $g \in \Gamma \mathcal{G}$ such that $\varphi(F)=gFg^{-1}$. And by definition of $K$, the automorphism $\left( \iota(g)^{-1} \varphi \right)_{|F}$ must be an inner automorphism of $F$, so there exists some $h \in F$ such that $\varphi(x)= ghxh^{-1}g^{-1}$ for every $x \in F$. In particular, the subgroups $\varphi(G_u)$ and $G_u$ are conjugate in $\Gamma \mathcal{G}$. 

\medskip \noindent
Thus, we have proved that $K$ is a subgroup of $\mathrm{ConjP}(\Gamma \mathcal{G})$. Because $K$ has finite index in $\mathrm{Aut}(\Gamma \mathcal{G})$, we conclude that $\mathrm{ConjP}(\Gamma \mathcal{G})$ is a finite-index subgroup of $\mathrm{Aut}(\Gamma \mathcal{G})$. As
$$\mathrm{ConjP}(\Gamma \mathcal{G})= \langle \text{partial conjugations} \rangle \rtimes \mathrm{Loc}(\Gamma \mathcal{G}),$$
we conclude that the subgroup generated by partial conjugations has finite index in $\mathrm{Aut}(\Gamma \mathcal{G})$ since $\mathrm{Loc}(\Gamma \mathcal{G})$ is clearly finite. 
\end{proof}

\noindent
As an interesting application of Theorem \ref{thm:GPfinitegroups}, we are able to determine precisely when the outer automorphism group of a graph product of finite groups is finite. Before stating the criterion, we need the following definition:

\begin{definition}
A \emph{Separating Intersection of Links} (or \emph{SIL} for short) in $\Gamma$ is the data of two vertices $u,v \in V(\Gamma)$ satisfying $d(u,v) \geq 2$ such that $\Gamma \backslash (\mathrm{link}(u) \cap \mathrm{link}(v))$ has a connected component which does not contain $u$ nor $v$.
\end{definition}

\noindent
Now we are able to state our criterion, which generalises \cite[Theorem 1.4]{AutGPabelian}:

\begin{cor}\label{cor:OutFinite}
The outer automorphism group $\mathrm{Out}(\Gamma \mathcal{G})$ is finite if and only if $\Gamma$ has no SIL.
\end{cor}

\noindent
The following argument was communicated to us by Olga Varghese.

\begin{proof}[Proof of Corollary \ref{cor:OutFinite}.]
Suppose that $\Gamma$ has no SIL. If $u,v \in V(\Gamma)$ are two distinct vertices, $a \in G_u$ and $b \in G_v$ two elements, and $\Lambda, \Xi$ two connected components of $\Gamma \backslash \mathrm{star}(u)$ and $\Gamma \backslash \mathrm{star}(v)$ respectively, then we claim that the two corresponding partial conjugations $(u, \Lambda,a)$ and $(v,\Xi,b)$ commute in $\mathrm{Out}(\Gamma \mathcal{G})$. For convenience, let $\Lambda_0$ (resp. $\Xi_0$) denote the connected component of $\Gamma \backslash \mathrm{star}(u)$ (resp. $\Gamma \backslash \mathrm{star}(v)$) which contains $v$ (resp. $u$). If $\Lambda \neq \Lambda_0$ or $\Xi \neq \Xi_0$, a direct computation shows that $(u, \Lambda,a)$ and $(v,\Xi,b)$ commute in $\mathrm{Aut}(\Gamma \mathcal{G})$ (see \cite[Lemma 3.4]{AutGPSIL} for more details). So suppose that $\Lambda= \Lambda_0$ and $\Xi=\Xi_0$. If $\Xi_1, \ldots, \Xi_k$ denote the connected components of $\Gamma \backslash \mathrm{star}(v)$ distinct from $\Xi_0$, notice that the product $(v,\Xi_0, b)(v, \Xi_1,b) \cdots (v,\Xi_k,b)$ is trivial in $\mathrm{Out}(\Gamma \mathcal{G})$ since the automorphism coincides with the conjugation by $b$. As we already know that $(u, \Lambda_0, a)$ commutes with $(v, \Xi_1, b) \cdots (v, \Xi_k,b)$ in $\mathrm{Aut}(\Gamma \mathcal{G})$, it follows that the following equalities hold in $\mathrm{Out}(\Gamma \mathcal{G})$:
$$\left[ (u, \Lambda_0, a) , (v, \Xi_0,b) \right]= \left[ (u, \Lambda_0,a) , (v, \Xi_k,b)^{-1} \cdots (v, \Xi_1,b)^{-1} \right] = 1.$$
This concludes the proof of our claim. Consequently, if $\mathrm{PC}(u)$ denotes the subgroup of $\mathrm{Out}(\Gamma \mathcal{G})$ generated by the (images of the) partial conjugations based at $u$ for every vertex $u \in V(\Gamma)$, then the subgroup $\mathrm{PC}$ of $\mathrm{Out}(\Gamma \mathcal{G})$ generated by all the (images of the) partial conjugations is naturally a quotient of $\bigoplus\limits_{u \in V(\Gamma)} \mathrm{PC}(u)$. But each $\mathrm{PC}(u)$ is finite; indeed, it has cardinality at most $c \cdot | G_u |$ where $c$ is the number of connected components of $\Gamma \backslash \mathrm{star}(u)$. Therefore, $\mathrm{PC}$ has to be finite. As $\mathrm{PC}$ has finite index in $\mathrm{Out}(\Gamma \mathcal{G})$ as a consequence of Theorem \ref{thm:GPfinitegroups}, we conclude that $\mathrm{Out}(\Gamma \mathcal{G})$ must be finite.

\medskip \noindent
Conversely, suppose that $\Gamma$ has a SIL. Thus, there exist two vertices $u,v \in V(\Gamma)$ satisfying $d(u,v) \geq 2$ such that $\Gamma \backslash (\mathrm{link}(u) \cap \mathrm{link}(v))$ has a connected component $\Lambda$ which contains neither $u$ nor $v$. Fix two non-trivial elements $a \in G_u$ and $b \in G_v$. A direct computation shows that the product $(u, \Lambda, a)(v, \Lambda, b)$ has infinite order in $\mathrm{Out}(\Gamma \mathcal{G})$. A fortiori, $\mathrm{Out}(\Gamma \mathcal{G})$ must be infinite.
\end{proof}

\begin{remark}
It is worth noticing that, if $\Gamma \mathcal{G}$ is a graph product of finite groups, then $\mathrm{ConjP}(\Gamma \mathcal{G})$ may be a proper subgroup of $\mathrm{Aut}(\Gamma \mathcal{G})$. For instance, 
$$\left\{ \begin{array}{ccc} a & \mapsto & ab \\ b & \mapsto & b \end{array} \right.$$
defines an automorphism of $\mathbb{Z}_2 \ast \mathbb{Z}_2 = \langle a \rangle \ast \langle b \rangle$ which does not belong to $\mathrm{ConjP}(\Gamma \mathcal{G})$. The inclusion $\mathrm{ConjP}(\Gamma \mathcal{G}) \subset \mathrm{Aut}(\Gamma \mathcal{G})$ may also be proper if $\Gamma$ is connected. For instance,
$$\left\{ \begin{array}{ccc} a & \mapsto & a \\ b & \mapsto & ab \\ c & \mapsto & c \end{array} \right.$$
defines an automorphism of $\mathbb{Z}_2 \oplus ( \mathbb{Z}_2 \ast \mathbb{Z}_2) = \langle a \rangle \oplus ( \langle b \rangle \ast \langle c \rangle)$ which does not belong to $\mathrm{ConjP}(\Gamma \mathcal{G})$.
\end{remark}

\subsection{Application to graph products over graphs of large girth}\label{section:applicationgirth}

\noindent
In this section, we focus on automorphism groups of graph products without imposing any restriction on their vertex groups. However, we need to impose more restrictive conditions on the underlying graph than in the previous sections. More precisely, the graphs which will interest us in the sequel are the following:

\begin{definition}
A \emph{molecular graph} is a finite connected simplicial graph without vertex of valence $<2$ and of girth at least $5$.
\end{definition}

\noindent
Molecular graphs generalise \emph{atomic graphs} introduced in \cite{RAAGatomic} (see Section \ref{section:atomic} below for a precise definition) by allowing separating stars. Typically, a molecular graph can be constructed from cycles of length at least five by gluing them together and by connecting them by trees. See Figure \ref{figure4} for an example. 
\begin{figure}[H]
\begin{center}
\includegraphics[trim={0 6cm 19cm 0},clip,scale=0.35]{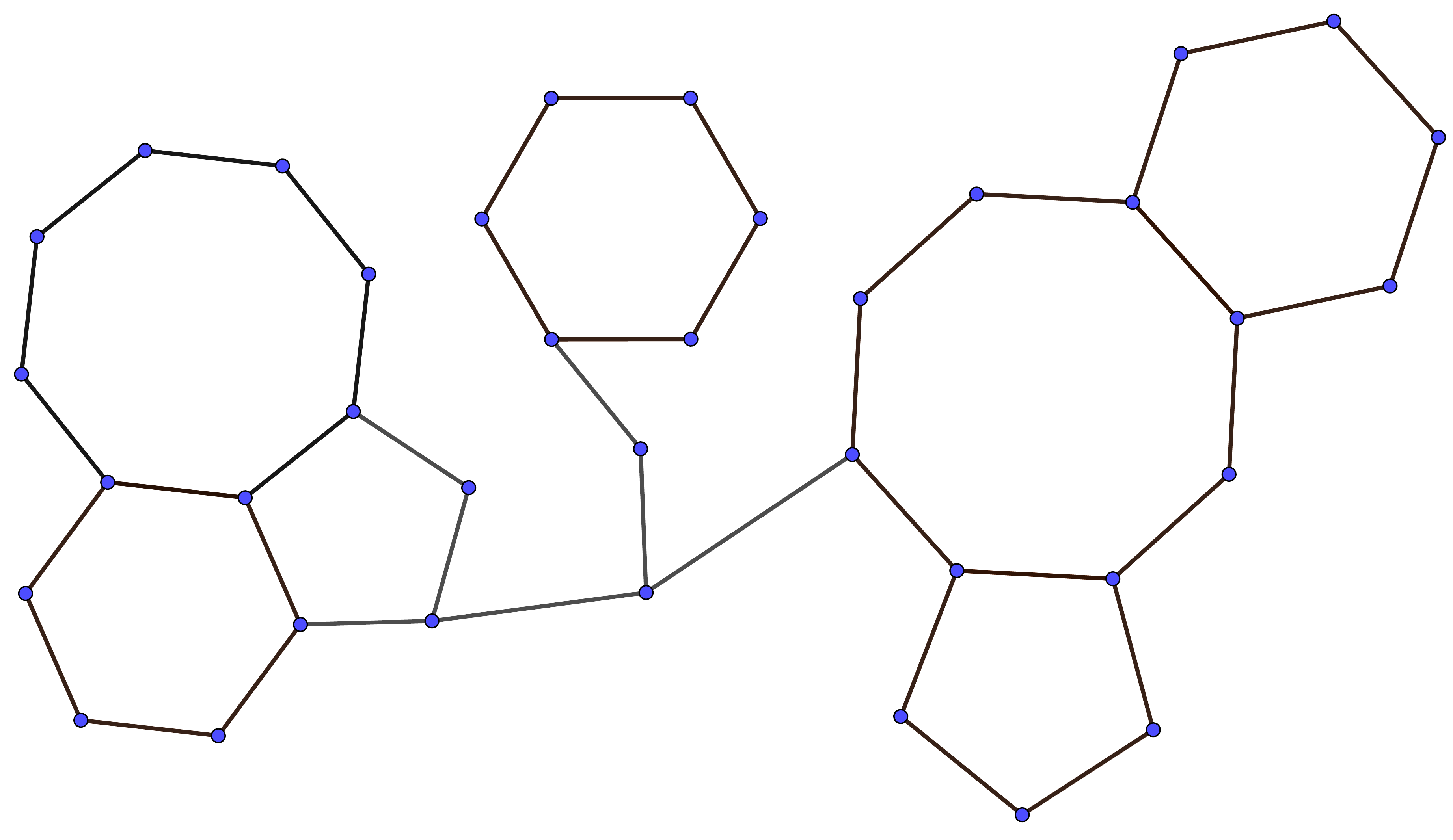}
\caption{A molecular graph.}
\label{figure4}
\end{center}
\end{figure}

\medskip \noindent
The main result of this section is the following statement, which widely extends \cite[Theorems A and D]{GPcycle}. It will be a consequence of Theorem \ref{thm:conjugatingauto}. 

\begin{thm}\label{thm:mainGPgeneral}
Let $\Gamma$ be a molecular graph and $\mathcal{G}$ a collection of groups indexed by $V(\Gamma)$. Then $\mathrm{Aut}(\Gamma \mathcal{G})= \mathrm{ConjP}(\Gamma \mathcal{G})$. 
\end{thm}

\noindent
Before turning to the proof of the theorem, we need to introduce some vocabulary.

\medskip \noindent
Given a simplicial graph $\Gamma$ and a collection of groups $\mathcal{G}$ indexed by $V(\Gamma)$, let $\mathcal{M}= \mathcal{M}(\Gamma, \mathcal{G})$ denote the collection of maximal subgroups of $\Gamma \mathcal{G}$ that decompose non trivially as direct products, and let $\mathcal{C}= \mathcal{C}(\Gamma, \mathcal{G})$ denote the collection of non trivial subgroups of $\Gamma \mathcal{G}$ that can be obtained by intersecting two subgroups of $\mathcal{M}$. A subgroup of $\Gamma \mathcal{G}$ that belongs to $\mathcal{C}$ is
\begin{itemize}
	\item \emph{$\mathcal{C}$-minimal} if it is minimal in $\mathcal{C}$ with respect to the inclusion;
	\item \emph{$\mathcal{C}$-maximal} if it is maximal in $\mathcal{C}$ with respect to the inclusion (or equivalently if it belongs to $\mathcal{M}$);
	\item \emph{$\mathcal{C}$-medium} otherwise.
\end{itemize}
It is worth noticing that these three classes of subgroups of $\Gamma \mathcal{G}$ are preserved by automorphisms. Now, we want to describe more explicitly the structure of theses subgroups. For this purpose, the following general result, which is a consequence of \cite[Corollary 6.15]{MinasyanOsinTrees}, will be helpful:

\begin{lemma}\label{lem:maxproduct}
Let $\Gamma$ be a simplicial graph and $\mathcal{G}$ a collection of groups indexed by $V(\Gamma)$. If a subgroup $H \leq \Gamma \mathcal{G}$ decomposes non-trivially as a product, then there exist an element $g \in \Gamma \mathcal{G}$ and a join $\Lambda \subset \Gamma$ such that $H \subset g \langle \Lambda \rangle g^{-1}$. 
\end{lemma}

\noindent
\textbf{For the rest of Section \ref{section:applicationgirth}, we fix a molecular graph $\Gamma$ and a collection of groups $\mathcal{G}$  indexed by $V(\Gamma)$.} The characterisation of the subgroups of $\mathcal{C}$ which we deduce from the previous lemma and from the quasi-median geometry of $X(\Gamma, \mathcal{G})$ is the following:

\begin{prop}\label{prop:SubgroupsC}
A subgroup $H \leq \Gamma \mathcal{G}$ belongs to $\mathcal{C}$ if and only if:
\begin{itemize}
	\item $H$ is conjugate to $\langle \mathrm{star}(u) \rangle$ for some vertex $u \in V(\Gamma)$; if so, $H$ is $\mathcal{C}$-maximal.
	\item Or $H$ is conjugate to $\langle G_u,G_v \rangle$ for some adjacent vertices $u,v \in V(\Gamma)$; if so, $H$ is $\mathcal{C}$-medium.
	\item Or $H$ is conjugate to $G_u$ for some vertex $u \in V(\Gamma)$; if so, $H$ is $\mathcal{C}$-minimal.
\end{itemize}
\end{prop}

\begin{proof}
Suppose that $H \leq \Gamma \mathcal{G}$ belongs to $\mathcal{M}$, ie., is a maximal product subgroup. It follows from Lemma \ref{lem:maxproduct} that there exist an element $g \in \Gamma \mathcal{G}$ and a join $\Lambda \subset \Gamma$ such that $H \subset g \langle \Lambda \rangle g^{-1}$. Because $\Gamma$ is triangle-free and square-free, $\Lambda$ must be the star of a vertex, say $\Lambda= \mathrm{star}(u)$ where $u \in V(\Gamma)$. As $g \langle \mathrm{star}(u) \rangle g^{-1}$ decomposes as a product, namely $g \left( G_u \times \langle \mathrm{link}(u) \rangle \right)g^{-1}$, it follows by maximality of $H$ that $H= g \langle \mathrm{star}(u) \rangle g^{-1}$.

\medskip \noindent
Conversely, we want to prove that, if $g \in \Gamma \mathcal{G}$ is an element and $u \in V(\Gamma)$ a vertex, then $g \langle \mathrm{star}(u) \rangle g^{-1}$ is a maximal product subgroup. So let $P$ be a subgroup of $\Gamma \mathcal{G}$ splitting non-trivially as a direct product and containing $g \langle \mathrm{star}(u) \rangle g^{-1}$. It follows from Lemma \ref{lem:maxproduct} that there exist an element $h \in \Gamma \mathcal{G}$ and a join $\Xi \subset \Gamma$ such that 
$$g \langle \mathrm{star}(u) \rangle g^{-1} \subset P \subset h \langle \Xi \rangle h^{-1}.$$
By applying Lemma \ref{lem:conjincluded}, we deduce that $\mathrm{star}(u) \subset \Xi$ and that $h \in g \langle \mathrm{star}(u) \rangle \cdot \langle \mathrm{link}(\mathrm{star}(u)) \rangle \cdot \langle \Xi \rangle$. As $\mathrm{star}(u)$ is a maximal join, necessarily $\mathrm{star}(u)= \Xi$ and $\mathrm{link}(\mathrm{star}(u))= \emptyset$. As a consequence, $h \in g \langle \mathrm{star}(u) \rangle$, so that
$$g \langle \mathrm{star}(u) \rangle g^{-1} \subset P \subset h \langle \Xi \rangle h^{-1}= h \langle \mathrm{star}(u) \rangle h^{-1} = g \langle \mathrm{star}(u) \rangle g^{-1}.$$
Therefore, $g \langle \mathrm{star}(u) \rangle g^{-1} = P$. 

\medskip \noindent
Thus, we have proved that 
$$\begin{array}{lcl} \mathcal{M} & = & \{ g \langle \mathrm{star}(u) \rangle g^{-1} \mid g \in \Gamma \mathcal{G}, u \in V(\Gamma) \} \\ \\ & = & \{ \mathrm{stab}(J) \mid \text{$J$ hyperplane of $X(\Gamma, \mathcal{G})$} \} \end{array}$$
where the last equality is justified by Theorem \ref{thm:HypStab}. As a consequence, the collection $\mathcal{C}$ coincides with the non-trivial subgroups of $\Gamma \mathcal{G}$ obtained by intersecting two hyperplane-stabilisers. Therefore, the implication of our lemma follows from the following observation:

\begin{fact}\label{fact:HypStabInter}
Let $J_1$ and $J_2$ be two hyperplanes of $X(\Gamma, \mathcal{G})$. 
\begin{itemize}
	\item If $J_1$ and $J_2$ are transverse, then $\mathrm{stab}(J_1) \cap \mathrm{stab}(J_2)$ is conjugate to $\langle G_u,G_v \rangle$ for some adjacent vertices $u,v \in V(\Gamma)$.
	\item If $J_1$ and $J_2$ are not transverse and if there exists a third hyperplane $J$ transverse to both $J_1$ and $J_2$, then $\mathrm{stab}(J_1) \cap \mathrm{stab}(J_2)$ is conjugate to $G_u$ for some vertex $u \in V(\Gamma)$.
	\item If $J_1$ and $J_2$ are not transverse and if no hyperplane is transverse to both $J_1$ and $J_2$, then $\mathrm{stab}(J_1) \cap \mathrm{stab}(J_2)$ is trivial. 
\end{itemize}
\end{fact}

\noindent
Suppose that $J_1$ and $J_2$ are transverse. As the carriers $N(J_1)$ and $N(J_2)$ intersect, say that $g \in X(\Gamma, \mathcal{G})$ belongs to their intersection, it follows that there exist two adjacent vertices $u,v \in V(\Gamma)$ such that $J_1=gJ_u$ and $J_2=gJ_v$. Therefore,
$$\mathrm{stab}(J_1) \cap \mathrm{stab}(J_2)= g \langle \mathrm{star}(u) \rangle g^{-1} \cap g \langle \mathrm{star}(v) \rangle g^{-1} = g \langle G_u,G_v \rangle g^{-1},$$
proving the first point of our fact. Next, suppose that $J_1$ and $J_2$ are not transverse but that there exists a third hyperplane $J$ transverse to both $J_1$ and $J_2$. As a consequence of Lemma \ref{lem:projhypsquaretrianglefree} below, we know that the projection of $N(J_1)$ onto $N(J_2)$, which must be $\mathrm{stab}(J_1) \cap \mathrm{stab}(J_2)$-invariant, is reduced to a single clique $C$. So $\mathrm{stab}(J_1) \cap \mathrm{stab}(J_2) \subset \mathrm{stab}(C)$. Notice that $J$ is dual to $C$. Indeed, the hyperplane dual to $C$ must be transverse to both $J_1$ and $J_2$, but we also know from Lemma \ref{lem:projhypsquaretrianglefree} below that there exists at most one hyperplane transverse to $J_1$ and $J_2$. We conclude from Proposition \ref{prop:rotstabinX} that
$$\mathrm{stab}(C)= \mathrm{stab}_{\circlearrowleft}(J) \subset \mathrm{stab}(J_1) \cap \mathrm{stab}(J_2),$$
proving that $\mathrm{stab}(J_1) \cap \mathrm{stab}(J_2) = \mathrm{stab}(C)$. Then, the second point of our fact follows from Lemma \ref{lem:CliqueStab}. Finally, suppose that $J_1$ and $J_2$ are not transverse and that no hyperplane is transverse to both $J_1$ and $J_2$. Then $\mathrm{stab}(J_1) \cap \mathrm{stab}(J_2)$ stabilises the projection of $N(J_2)$ onto $N(J_1)$, which is reduced to a single vertex. As vertex-stabilisers are trivial, the third point of our fact follows. 

\medskip \noindent
Conversely, if $u,v \in V(\Gamma)$ are two adjacent vertices, then $\langle G_u , G_v \rangle$ is the intersection of $\langle \mathrm{star}(u) \rangle$ and $\langle \mathrm{star}(v) \rangle$; and if $w \in V(\Gamma)$ is a vertex, then $G_w$ is the intersection of $\langle \mathrm{star}(x) \rangle$ and $\langle \mathrm{star}(y) \rangle$ where $x,y \in V(\Gamma)$ are two distinct neighbors of $w$. 
\end{proof}

\noindent
The following result is used in the proof of Proposition \ref{prop:SubgroupsC}. 

\begin{lemma}\label{lem:projhypsquaretrianglefree}
Let $J_1,J_2$ two non-transverse hyperplanes of $X(\Gamma, \mathcal{G})$. There exists at most one hyperplane transverse to both $J_1$ and $J_2$. As a consequence, the projection of $N(J_1)$ onto $N(J_2)$ is either a single vertex or a single clique.
\end{lemma}

\begin{proof}
Because $\Gamma$ is triangle-free, we know from Corollary \ref{cor:cubdim} that the cubical dimension of $X(\Gamma)$ is two. Consequently, if there exist two hyperplanes transverse to both $J_1$ and $J_2$, they cannot be transverse to one another. So, in order to conclude that at most one hyperplane may be transverse to both $J_1$ and $J_2$, it is sufficient to prove the following observation:

\begin{claim}\label{claim:lengthfour}
The transversality graph $T(\Gamma, \mathcal{G})$ does not contain an induced cycle of length four.
\end{claim}

\noindent
Suppose by contradiction that $X(\Gamma, \mathcal{G})$ contains a cycle of four hyperplanes $(J_1, \ldots, J_4)$. Suppose that the quantity $d(N(J_1),N(J_3))+d(N(J_2),N(J_4))$ is minimal. If $N(J_1)$ and $N(J_3)$ are disjoint, they must be separated by some hyperplane $J$. Replacing $J_1$ with $J$ produces a new cycle of four hyperplanes of lower complexity, contradicting the choice of our initial cycle. Therefore, $N(J_1)$ and $N(J_3)$ must intersect. Similarly, $N(J_2) \cap N(J_4) \neq \emptyset$. Thus, $N(J_1), \ldots, N(J_4)$ pairwise intersect, so that there exists a vertex $x \in N(J_1) \cap \cdots \cap N(J_4)$. For every $1 \leq i \leq 4$, let $u_i$ denote the vertex of $\Gamma$ labelling the hyperplane $J_i$. Notice that, as a consequence of Lemma \ref{lem:transverseimpliesadj}, $u_1 \neq u_3$ (resp. $u_2 \neq u_4$) because $J_1$ and $J_3$ (resp. $J_2$ and $J_4$) are tangent. So $u_1, \ldots, u_4$ define a cycle of length four in $\Gamma$, contradicting the fact that $\Gamma$ is molecular.

\medskip \noindent
This concludes the proof of the first assertion of our lemma. The second assertion then follows from Corollary \ref{cor:diamproj}. 
\end{proof}

\noindent
Now, Theorem \ref{thm:mainGPgeneral} is clear since it follows from Proposition \ref{prop:SubgroupsC} that Theorem \ref{thm:conjugatingauto} applies, proving that $\mathrm{Aut}(\Gamma \mathcal{G})= \mathrm{ConjAut}(\Gamma \mathcal{G})= \mathrm{ConjP}(\Gamma \mathcal{G})$. In fact, we are able to prove a stronger statement, namely:

\begin{thm}
Let $\Gamma, \Phi$ be two molecular graphs and $\mathcal{G},\mathcal{H}$ two collections of groups indexed by $V(\Gamma), V(\Phi)$ respectively. Suppose that there exists an isomorphism $\varphi : \Gamma \mathcal{G} \to \Phi \mathcal{H}$. Then there exist an automorphism $\alpha \in \mathrm{ConjP}(\Gamma \mathcal{G})$ and an isometry $s : \Gamma \to \Phi$ such that $\varphi \alpha$ sends isomorphically $G_u$ to $H_{s(u)}$ for every $u \in V(\Gamma)$.
\end{thm}

\begin{proof}
It follows from Proposition \ref{prop:SubgroupsC} that conjugates of vertex groups may be defined purely algebraically, so that the isomorphism $\varphi : \Gamma \mathcal{G} \to \Phi \mathcal{H}$ must send vertex groups of $\Gamma \mathcal{G}$ to conjugates of vertex groups of $\Phi \mathcal{H}$. Then Theorem \ref{thm:ConjIsom} applies, providing the desired conclusion. 
\end{proof}

\section{Graph products of groups over atomic graphs}\label{section:atomic}

\noindent
In this section, we focus on automorphism groups of graph products over a specific class of simplicial graphs, namely \emph{atomic graphs}. Recall from \cite{RAAGatomic} that a finite simplicial graph is \emph{atomic} if it is connected, without vertex of valence $<2$, without separating stars, and with girth $\geq 5$. In other words, atomic graphs are molecular graphs without separating stars. As a consequence of Theorem \ref{thm:mainGPgeneral}, the automorphism group of a graph product over an atomic graph is given by  $\mathrm{Aut}(\Gamma\mathcal{G})= \langle \mathrm{Inn}(\Gamma\mathcal{G}), \mathrm{Loc}(\Gamma \mathcal{G}) \rangle$. 

\begin{lem}\label{fact_2}
Let $\Gamma$ be a simplicial graph and $\mathcal{G}$ a collection of groups indexed by $V(\Gamma)$. If $\Gamma$ is not the star of a vertex, then 	$\mathrm{Inn}(\Gamma \mathcal{G}) \cap \mathrm{Loc}(\Gamma \mathcal{G})= \{ \mathrm{Id} \}$.
\end{lem}

\begin{proof}
	Let $g \in \Gamma\cG$, let $\varphi\in \mathrm{Loc}(\Gamma \mathcal{G})$, let $\sigma$ be the automorphism of $\Gamma$ induced by $\varphi$, and suppose that $\iota(g) = \varphi$. Then in particular for every vertex $v$ of $\Gamma$, we have $gG_vg^{-1} = \varphi(G_v) = G_{\sigma(v)}.$ Since distinct local groups are not conjugated under $\Gamma\cG$, it follows that $\sigma(v) = v$, and hence $g$ normalises $G_v$. By standard results on graph products, this implies that $g \in \langle \mathrm{star}(v) \rangle$. As this holds for every vertex of $\Gamma$, we get $g \in \cap_v  \langle \mathrm{star}(v) \rangle$. Since $\Gamma$ is not the star of a vertex, it follows that $\cap_v  \langle \mathrm{star}(v) \rangle = \{1\}$, hence $\varphi = \iota(g)=\mathrm{Id}$. 
\end{proof}

Thus, the automorphism group of a graph product over an atomic graph is given by  $\mathrm{Aut}(\Gamma\mathcal{G})=  \mathrm{Inn}(\Gamma\mathcal{G})\rtimes \mathrm{Loc}(\Gamma \mathcal{G}) $. In particular, any automorphism of $\Gamma\cG$ decomposes in a unique way as a product of the form $ \iota(g) \varphi $ with $g \in \Gamma\cG, \varphi \in \mathrm{Loc}(\Gamma \mathcal{G}) $. 

Recall that the Davis complex of a graph product is a CAT(0) cube complex that was introduced in Definition \ref{def:Davis}. The fundamental observation is that in the case of atomic graphs, the automorphism group of $\Gamma\cG$ acts naturally on the associated Davis complex. 


\begin{lem}\label{prop:extendT}
Let $\Gamma$ be an atomic graph and $\mathcal{G}$ a collection of groups indexed by $V(\Gamma)$.  
The action of $\Gamma \mathcal{G}$ on the Davis complex $D(\Gamma, \mathcal{G})$ extends to an action $\mathrm{Aut}(\Gamma \mathcal{G}) \curvearrowright D(\Gamma, \mathcal{G})$, where $\Gamma \mathcal{G}$ is identified canonically with $\mathrm{Inn}(\Gamma \mathcal{G})$. More precisely, the action is given by 
$$ \iota(g) \varphi \cdot hH \coloneqq g\varphi(h)\varphi(H),$$
for $g, h \in \Gamma\cG$, $H$ a  subgroup of $\Gamma\cG$ of the form $\langle \Lambda \rangle$ for some complete subgraph $\Lambda$, and $\varphi \in \mathrm{Loc}(\Gamma \mathcal{G})$.
\end{lem}

\begin{proof} By definition, elements of $\mathrm{Loc}(\Gamma \mathcal{G})$ preserve the family of subgroups of the form $\langle \Lambda \rangle$ for some complete subgraph $\Lambda$, so the action is well defined at the level of the vertices of $D(\Gamma, \cG)$. By definition of the edges and higher cubes of the Davis complex, one sees that the action extends to an action on $D(\Gamma, \cG)$ itself. One checks easily from the definition that the restriction to Inn$(\Gamma\cG)$ coincides with the natural action of $\Gamma\cG$ on $D(\Gamma, \cG)$ by left multiplication. 
\end{proof}

\noindent
As a first application of Proposition \ref{prop:extendT}, we are able to show that the automorphism group of a graph product over an atomic graph does not satisfy Kazhdan's property (T). 

\begin{thm}\label{thm:noT}
Let $\Gamma$ be an atomic graph and $\mathcal{G}$ a collection of groups indexed by $V(\Gamma)$. Then the automorphism group $\mathrm{Aut}(\Gamma \mathcal{G})$ does not satisfy Kazhdan's property~(T).
\end{thm}

\begin{proof}
The action of $\Gamma \mathcal{G}$ on $D(\Gamma, \mathcal{G})$ is non-trivial by construction, hence has unbounded orbits, so in particular the action of Aut$(\Gamma \mathcal{G})$ on the CAT(0) cube complex  $D(\Gamma, \mathcal{G})$ has unbounded orbits. The result then follows from a theorem of Niblo-Roller  \cite{MR1459140}.
\end{proof}

\noindent
As a second application, we prove the following: 

\begin{thm}\label{thm:acyl_hyp}
Let $\Gamma$ be an atomic graph and $\mathcal{G}$ a collection of finitely generated groups indexed by $V(\Gamma)$. 
Then the automorphism group $\mathrm{Aut}(\Gamma  \mathcal{G})$ is acylindrically hyperbolic.
\end{thm}

\noindent \textbf{For the rest of Section \ref{section:atomic}, we fix an atomic  graph $\Gamma$ and a collection of groups $\mathcal{G}$  indexed by $V(\Gamma)$.}  To prove Theorem \ref{thm:acyl_hyp}, we use a criterion introduced in \cite{NoteAcylHyp} to show the acylindrical hyperbolicity of a group via its action on a CAT(0) cube complex. Following \cite{CapraceSageev}, we say that the action of a group $G$ on a CAT(0) cube complex $Y$ is \textit{essential} if no $G$-orbits stay in some neighborhood of a half-space. Following \cite{ChatterjiFernosIozzi}, we say that the action is \textit{non-elementary} if $G$ does not have a finite orbit in $Y \cup \partial_{\infty}Y$.  We further say that $Y$ is \textit{cocompact} if its automorphism group acts cocompactly on it; that it is \textit{irreducible} if it does not split as the direct product of two non-trivial CAT(0) cube complexes; and that it \textit{does not have a free face} if every non-maximal cube is contained in at least two maximal cubes. 

\begin{thm}[{\cite[Theorem 1.5]{NoteAcylHyp}}]\label{thm:Chatterji_Martin}
	Let $G$ be a group acting  non-elementarily and essentially on a  finite-dimensional irreducible cocompact CAT(0) cube complex $X$ with no free face. If there exist two points $x,y \in X$ such that $\mathrm{stab}(x) \cap \mathrm{stab}(y)$ is finite, then $G$ is acylindrically hyperbolic.
\end{thm}

We will use this criterion for the action of Aut$(\Gamma\cG)$ on the Davis complex $D(\Gamma, \cG)$. To this end, we need to check a few things about the action. 

\begin{lem}\label{lem:ess}
	The action of $\mathrm{Aut}(\Gamma\cG)$ on $D(\Gamma, \cG)$ is essential and non-elementary.
\end{lem}

\begin{proof}
	It is enough to show that the action of $\Gamma \cG$ (identified with the subgroup  $\mathrm{Inn}(\Gamma\cG)$ of $\mathrm{Aut}(\Gamma\cG)$) acts essentially and non-elementarily on $D(\Gamma, \cG)$. 
	
	\textbf{Non-elementarity.} The Davis complex is quasi-isometric to a building with underlying Coxeter group the right-angled Coxeter group $W_\Gamma$ (see Section \ref{sec:dual}). As $\Gamma$ has girth at least $5$, $W_\Gamma$ is hyperbolic, and thus so is $D(\Gamma, \cG)$. Since the Davis complex is hyperbolic, the non-elementarity of the action will follow from the fact that there exist two elements $g, h \in \Gamma\cG$ acting hyperbolically on $D(\Gamma, \cG)$ and having disjoint limit sets in the Gromov boundary of $D(\Gamma, \cG)$, by elementary considerations of the dynamics of the action on the boundary of a hyperbolic space. We now construct such hyperbolic elements, following methods used in \cite{NoteAcylHyp} and \cite{CharneyMorrisArtinAH}.
	
	Since $\Gamma$ is leafless and has girth at least $5$, we can find an induced  geodesic of $\Gamma$ of the form $v_1, \ldots, v_4$. For each $1\leq i\leq 4$, choose a non-trivial element $s_i \in G_{v_i}$. For each $1\leq i\leq 4$, let also $e_{v_i}$ be the edge of $D(\Gamma, \cG)$ between the coset $\langle \varnothing \rangle $ and the coset $\langle v_i \rangle$. 
	
	We define the following two elements: $g = s_3s_1$ and $h = s_4s_2,$ 
	as well as the following combinatorial paths of $D(\Gamma, \cG)$: 
	$$\Lambda_g =  \bigcup_{k\in \bbZ} g^k (e_{v_1}\cup e_{v_3} \cup s_3 e_{v_3} \cup s_3e_{v_1}),~~~~~  \Lambda_h =  \bigcup_{k\in \bbZ} h^k (e_{v_2}\cup e_{v_4} \cup s_4 e_{v_4} \cup s_4e_{v_2}).$$
	
	We claim that $\Lambda_g$ and $\Lambda_h$ are combinatorial axes for $g$ and $h$ respectively. Indeed, by construction $\Lambda_g$ and $\Lambda_h$ make an angle $\pi$ at each vertex, hence are geodesic lines, and each element acts by translation on its associated path by construction.
	
	Now notice that no hyperplane of $D(\Gamma, \cG)$ crosses both $\Lambda_g$ and $\Lambda_h$,  since edges in $\Lambda_g$ and $\Lambda_h$ are disjoint and edges defining the same hyperplane of $D(\Gamma, \cG)$ have the same label.
	 This implies that for every vertex $u$ of $\Lambda_g$, the (unique) geodesic between $u$ and $\Lambda_h$ is the sub-path of $\Lambda_g$ between $u$ and the intersection point $\Lambda_g \cap \Lambda_h$. This in turn implies that the limit sets of $\Lambda_g$, $\Lambda_{h}$ are disjoint, and that the hyperplanes associated to the edges $e_{v_2}, e_{v_3}$ (which contain $\Lambda_g$, $\Lambda_{h}$ respectively) are essential. 
	
	\textbf{Essentiality.}   For every vertex $v$ of $\Gamma$, we can use the fact that $\Gamma$ is leafless and of girth at least $5$  to construct a geodesic path $v_1, \ldots, v_4$ of $\Gamma$ with $v_2=v$. The above reasoning implies that the hyperplane associated to $e_v$ is essential. As every hyperplane of $D(\Gamma, \cG)$ is a translate of such a hyperplane, it follows that  hyperplanes of $D(\Gamma, \cG)$  are essential. As the action of $\Gamma \cG$ on $D(\Gamma, \cG)$ is cocompact, the action of  $\mathrm{Aut}(\Gamma\cG)$ on $D(\Gamma, \cG)$ is also essential.
\end{proof}

\begin{lem}\label{lem:irred}
	The Davis complex $D(\Gamma, \cG)$ is irreducible.
\end{lem}

\begin{proof}
	The link of every vertex of $D(\Gamma, \cG)$ corresponding to a coset of the trivial subgroup has a link which is isomorphic to $\Gamma$. As $\Gamma$ has girth at least $5$, such a link does not decompose non-trivially as a join, hence $D(\Gamma, \cG)$ does not decompose non-trivially as a direct product.
\end{proof}

\begin{lem}\label{lem:no_free_face}
	The Davis complex $D(\Gamma, \cG)$ has no free face. 
\end{lem}

\begin{proof}
	Since $\Gamma$ has girth at least $5$, the Davis complex is $2$-dimensional, and we have to show that every edge is contained in at least two squares.  Let $e$ be an edge of $D(\Gamma, \cG)$. If $e$ contains a vertex $v$ that is a coset of the trivial subgroup, then the link of $v$ is isomorphic to $\Gamma$, and it follows from the leafless-ness assumption that $e$ is contained in at least two squares. Otherwise  let $C$ be a any square containing $e$. As $\mathrm{Stab}_{\Gamma\cG}(C)$ is trivial and $\mathrm{Stab}_{\Gamma\cG}(e)$, which is conjugate to some $G_v$, contains at least two elements, it follows that there are at least two squares containing $e$. 
\end{proof}

\begin{lem}\label{lem:stab}
	Let $P$ be the fundamental domain of $D(\Gamma, \cG)$ and let $g \in \Gamma\cG$. Then 
	$$ \mathrm{Stab}_{\mathrm{Aut}(\Gamma\cG)}(P)\cap  \mathrm{Stab}_{\mathrm{Aut}(\Gamma\cG)}(gP) = \{ \varphi \in \mathrm{Loc}(\Gamma \cG) \mid \varphi(g)=g\}.$$	
\end{lem}

\begin{proof}
	Since an element of $\mathrm{Aut}(\Gamma\cG)$ stabilises the fundamental domain of $D(\Gamma, \cG)$ if and only if it stabilises the vertex corresponding to the coset $\langle \varnothing \rangle$, we have  $\mbox{Stab}_{\mathrm{Aut}(\Gamma\cG)}(P) = \mbox{Stab}_{\mathrm{Aut}(\Gamma\cG)}(\langle \varnothing \rangle)  = \mathrm{Loc}(\Gamma \cG)$. Therefore, if $\varphi \in \mathrm{Aut}(\Gamma\cG)$ belongs to $\mbox{Stab}_{\mathrm{Aut}(\Gamma\cG)}(P)\cap  \mbox{Stab}_{\mathrm{Aut}(\Gamma\cG)}(gP)$ then $\varphi \in  \mathrm{Loc}(\Gamma \cG)$ and there exists some $\psi \in \mathrm{Loc}(\Gamma \cG)$ such that $\varphi = \iota(g) \circ \psi \circ \iota(g)^{-1}$. Since $\psi \circ \iota(g)^{-1} = \iota(\psi(g))^{-1} \circ \psi$, we deduce that
	$$\varphi \circ \psi^{-1}= \iota(g) \circ \psi \circ \iota(g)^{-1} \circ \psi^{-1}  = \iota(g) \circ \iota(\psi(g))^{-1} \in \mathrm{Inn}(\Gamma\cG) .$$
	Thus, $\varphi \circ \psi^{-1} \in \mathrm{Inn}(\Gamma\cG) \cap \mathrm{Loc}(\Gamma \cG).$ On the other hand, we know from Lemma \ref{fact_2} that $\mathrm{Inn}(\Gamma\cG) \cap \mathrm{Loc}(\Gamma \cG)= \{ \mathrm{Id} \}$, whence $\varphi= \psi$ and $\iota(g)= \iota(\psi(g))$. 
As a consequence of \cite[Theorem 3.34]{GreenGP}, it follows from the fact that $\Gamma$ has girth at least $5$ that $\Gamma\cG$ is centerless, so that the equality $\iota(g)= \iota(\psi(g))$ implies that $\varphi(g)=g$, hence  the inclusion $\mbox{Stab}_{\mathrm{Aut}(\Gamma\cG)}(P)\cap  \mbox{Stab}_{\mathrm{Aut}(\Gamma\cG)}(gP) \subset \{ \varphi \in \mbox{Loc}(\Gamma \cG) \mid \varphi(g)=g\}.$
	The reverse inclusion is clear.
\end{proof}

\begin{proof}[Proof of Theorem \ref{thm:acyl_hyp}.]
	For each vertex $v$ of $\Gamma$, choose a finite generating set $\{s_{v,j} \mid 1 \leq j \leq m_v\}$. Up to allowing repetitions, we will assume that all the integers $m_v$ are equal, and denote by $m$ that integer. We now define a specific element $g \in \Gamma \mathcal{G}$ in the following way.
	
	Since $\Gamma$ is leafless and has girth at least $5$, we can find a sequence $v_1, \ldots, v_n$ exhausting all the vertices of $\Gamma$, such that for each $1\leq i < n$, $v_i$ and $v_{i+1}$ are not adjacent, and also $v_n$ and $v_1$ are not adjacent. We now define: 
	$$g_{j}:=g _{v_1, j}\cdots g_{v_n,j} \mbox{ for $1 \leq j \leq m$,}$$
	$$g:=g_{1}\cdots  g_m.$$
	Let $\varphi$ be an element of $ \mbox{Stab}_{\mathrm{Aut}(\Gamma\cG)}(P)\cap  \mbox{Stab}_{\mathrm{Aut}(\Gamma\cG)}(gP)$. By Lemma \ref{lem:stab}, it follows that $\varphi \in \mathrm{Loc}(\Gamma \cG)$ and $\varphi(g) = g$. By construction, $g$ can be written as a concatenation of the form $g= s_1 \cdots s_p$, where each $s_k$ is of the form $s_{v,j}$, and such that no consecutive $s_k, s_{k+1}$ belong to groups of $\cG$ that are joined by an edge of $\Gamma$. In particular, it follows from the properties of normal forms recalled in Section \ref{sec:generalities} that the decomposition $g= s_1 \cdots s_p$ is the unique reduced form of $g$. As $g = \varphi(g) = \varphi(s_1)\cdots\varphi(s_p)$ is an another reduced form of $g$, it follows that $\varphi(s_k)=s_k$ for every $1 \leq k \leq p$. As this holds for a generating set of $\Gamma\cG$, it follows that  $\varphi$ is the identity. We thus have that $ \mbox{Stab}_{\mathrm{Aut}(\Gamma\cG)}(P)\cap  \mbox{Stab}_{\mathrm{Aut}(\Gamma\cG)}(gP)$ is trivial. It now follows from Lemmas  \ref{lem:ess}, \ref{lem:irred}, \ref{lem:no_free_face}, and \ref{lem:stab} that Theorem \ref{thm:Chatterji_Martin} applies, hence Aut$(\Gamma\cG)$ is acylindrically hyperbolic.
\end{proof}

\begin{remark}
It is worth noticing that, in the statement of Theorem \ref{thm:acyl_hyp}, we do not have to require our vertex groups to be finitely generated. Indeed, during the proof, we only used the following weaker assumption: every $G \in \mathcal{G}$ contains a finite set $S \subset G$ such that every automorphism of $G$ fixing pointwise $S$ must be the identity. Of course, if $G$ is finitely generated, we may take $S$ to be  a finite generating set. But if $G= \mathbb{Q}$ for instance, then $S= \{1\}$ works as well, even though $\mathbb{Q}$ is not finitely generated. However, some condition is needed, as shown by the example below.

\medskip \noindent
Let $Z$ be the direct sum $\bigoplus\limits_{p \ \text{prime}} \mathbb{Z}_p$ and let $G_n$ be the graph product of $n$ copies of $Z$ over the cycle $C_n$ of length $n \geq 5$. We claim that $\mathrm{Aut}(G_n)$ is not acylindrically hyperbolic. 

\medskip \noindent
As a consequence of Corollary C (stated in the introduction), the automorphism group $\mathrm{Aut}(G_n)$ decomposes as $\left( \mathrm{Inn}(G_n) \rtimes \mathrm{Loc}(C_n, \mathcal{G}) \right) \rtimes \mathrm{Sym}(C_n, \mathcal{G})$. As the property of being acylindrically hyperbolic is stable under taking infinite normal subgroups \cite[Lemma 7.2]{OsinAcyl}, it is sufficient to show that $\mathrm{Inn}(G_n) \rtimes \mathrm{Loc}(C_n, \mathcal{G})$ is not acylindrically hyperbolic. So let $\iota(g) \varphi$ be an automorphism of this subgroup, where $g \in G_n$ and $\varphi \in \mathrm{Loc}(C_n, \mathcal{G})$. For each copy $Z_i$ of $Z$, the reduced word representing $g$ contains only finitely many syllables in $Z_i$; let $S_i \subset Z_i$ denote this set of syllables. Clearly, there exists an infinite collection of automorphisms of $Z_i$ fixing $S_i$ pointwise; furthermore, we may suppose that this collection generates a subgroup of automorphisms $\Phi_i \leq \mathrm{Aut}(Z_i)$ which is a free abelian group of infinite rank. Notice that $\phi(g)=g$ for every $\phi \in \Phi_i$. Therefore, for every $\psi \in \Phi_1 \times \cdots \times \Phi_n \leq \mathrm{Loc}$, we have
\begin{center}
$\psi \cdot \iota(g) \varphi= \iota(\psi(g)) \cdot \psi \varphi=g \cdot \psi \varphi= g \cdot \varphi \psi = g \varphi \cdot \psi$,
\end{center}
since $\varphi$ and $\psi$ clearly commute: each $\mathrm{Aut}(Z_i)$ is abelian so that $\mathrm{Loc}(C_n, \mathcal{G})$ is abelian as well. Thus, we have proved that the centraliser of any element of $\mathrm{Inn}(G_n) \rtimes \mathrm{Loc}(C_n, \mathcal{G})$ contains a free abelian group of infinite rank. Therefore, $\mathrm{Inn}(G_n) \rtimes \mathrm{Loc}(C_n, \mathcal{G})$ (and a fortiori $\mathrm{Aut}(G_n)$) cannot be acylindrically hyperbolic according to \cite[Corollary~6.9]{OsinAcyl}.
\end{remark}

\begin{remark}
In this section, it was more convenient to work with a CAT(0) cube complex rather than with a quasi-median graph because results already available in the literature allowed us to shorten the arguments. However, we emphasize that a quasi-median proof is possible. The main steps are the followings. First, as in Lemma \ref{prop:extendT}, the action $\Gamma \mathcal{G} \curvearrowright X(\Gamma, \mathcal{G})$ extends to an action $\mathrm{Aut}(\Gamma \mathcal{G}) \curvearrowright X(\Gamma, \mathcal{G})$ via $\iota(g) \varphi \cdot x = g \varphi(x)$ where $x \in X(\Gamma, \mathcal{G})$ is a vertex. So Theorem \ref{thm:noT} also follows since Niblo and Roller's argument \cite{MR1459140} can be reproduced almost word for word in the quasi-median setting; or alternatively, the theorem follows from the combination of \cite{MR1459140} with \cite[Proposition 4.16]{Qm} which shows that any quasi-median graph admits a ``dual'' CAT(0) cube complex. Next, in order to prove Theorem \ref{thm:acyl_hyp}, we need to notice that $X(\Gamma, \mathcal{G})$ is hyperbolic (as a consequence of \cite[Fact 8.33]{Qm}) and that the element $g$ constructed in the proof above turns out to be a WPD element. For the latter observation, one can easily prove the following criterion by following the arguments of \cite[Theorem 18]{MoiAcylHyp}: given a group $G$ acting on a hyperbolic quasi-median graph $X$, if an element $g \in G$ skewers a pair of strongly separated hyperplanes $J_1,J_2$ such that $\mathrm{stab}(J_1) \cap \mathrm{stab}(J_2)$ is finite, then $g$ must be a WPD element. 
\end{remark}

\addcontentsline{toc}{section}{References}

\bibliographystyle{alpha}
{\footnotesize\bibliography{ConjAutGP}}

\Address

\end{document}